\documentclass[12pt]{article}

\usepackage{fullpage,amsfonts,amsmath,amsthm,amssymb,mathrsfs,graphicx,enumitem}

\usepackage{verbatim,url}

   \usepackage[top=2.54cm, bottom=2.54cm, left=2.54 cm, right=2.54 cm]{geometry}
 \newcommand{\lab}[1]{\label{#1}}                % hides labels

 \usepackage[usenames,dvipsnames]{color}
\newcommand{\remove}[1]{}
\newcommand\eqn[1]{(\ref{#1})}

\newcommand{\be}{\begin{equation}}
\newcommand{\bel}[1]{\begin{equation}\lab{#1}\ }
\newcommand{\ee}{\end{equation}}
\newcommand{\bea}{\begin{eqnarray}}
\newcommand{\eea}{\end{eqnarray}}
\newcommand{\bean}{\begin{eqnarray*}}
\newcommand{\eean}{\end{eqnarray*}}

\newtheorem{thm}{Theorem}%[section]

\newtheorem{lemma}[thm]{Lemma}

\def\proof{\noindent{\bf Proof.\ }  }
\def\qed{~~\vrule height8pt width4pt depth0pt}

\newcommand{\UB}{{\overline m}}
\newcommand{\LB}{{\underline m}}

\def\ex{{\mathbb E}}
\def\pr{{\mathbb P}}
 \def\P{{\cal P}}
\def\C{{\cal S}}
\def\state{{\cal S}}

 \def\reject{{R}}
\def\output{{F}}
\def\markov{\mathscr{M}}
\def\classes{\mathcal{C}}
\def\types{\mathcal{T}}
\def\Alpha{\classes}
\def\Tau{\types}
\def\group{stratum}
\def\groups{strata}

\def\quads{\mathscr{S}}
\def\acceptableMW{\mathcal{A}}
\def\acceptable{\mathcal{A}_{\gamma}}  % This was called \Phi_0 previously; it's the set that is not initially rejected
\def\simple{\mathcal{B}}
\def\imax{i_1}
\def\vs{{\vec \sigma}}

\newcommand{\NameA}{{\em REG}}
\newcommand{\NameB}{{\em REG*}}
\def\eps{\epsilon}

\def\ss{\smallskip}
\def\non{\nonumber}
\def\no{\noindent}

\date{}

\title{Uniform generation of random regular graphs\footnote{The conference version will appeared in {\em Proc.\ FOCS 2015}.}}
\author{
Pu Gao\thanks{Research supported by NSERC while this author was affiliated with University of Waterloo. }\\
 Monash University\\
jane.gao@monash.edu
\and
Nicholas Wormald\thanks{Research supported by  Australian Laureate Fellowships grant FL120100125.}
\\
Monash University\\
nick.wormald@monash.edu }
\begin{document}
\maketitle
\begin{abstract}
We develop a new approach for uniform generation of combinatorial objects,  and apply  it to derive a uniform sampler \NameA\ for $d$-regular graphs. \NameA\ can be implemented such that each graph is generated in expected time  $O(nd^3)$, provided that $d=o(\sqrt{n})$. Our result significantly improves the previously best uniform sampler, which works efficiently only when $d=O(n^{1/3})$, with essentially the same running time for the same $d$. We also give a linear-time approximate sampler \NameB, which generates a random $d$-regular graph whose distribution differs from the uniform by $o(1)$ in total variation distance, when $d=o(\sqrt{n})$.

\end{abstract}

\section{Introduction}
\lab{s:intro}
 
Research on uniform generation of random graphs is almost as old as modern computing. Tinhofer~\cite{Tinhofer} gave a generation algorithm in which the probabilities of the graphs are computed {\em a posteriori}. In theory, this could be used to produce any desired distribution by rejection sampling, but  no explicit bounds on the time complexity of this method are known. The earliest method useful in practice for achieving exactly uniform generation arose from the enumeration methods of  B{\'e}k{\'e}ssy,  B{\'e}k{\'e}ssy and Koml\'os~\cite{BBK},  Bender and Canfield~\cite{BC}  and Bollob\'{a}s~\cite{B}. This works  in linear expected time for graphs with bounded maximum degree $d$, though the multiplicative constant behaves like $e^{ (d-1)^2/4}$ as a function of $d$ (making it rather impractical for moderately large $d$, even $d=12$).  It generalises easily to a simple algorithm for uniform generation of graphs with given degrees. (See, for example,~\cite{W84}, which in addition gives an algorithm for 3-regular graphs that has linear deterministic time.) The algorithm starts by generating a pairing, to be defined below, uniformly at random. If the corresponding graph is simple then it is outputted. Otherwise, the algorithm is restarted.

A big advance on exactly uniform generation, which significantly relaxed the constraint on the maximum degree, was by McKay and Wormald~\cite{MWgen}. Their algorithm efficiently and uniformly generates random graphs with given degrees as long as the maximum degree is $O(M^{1/4})$ where $M$ is the total degree. A   case of particular interest is the generation of $d$-regular graphs, where the expected running time is $O(nd^3)$, for any $d=O(n^{1/3})$. This is currently the best exactly uniform sampler for regular graphs and graphs with given degrees.

When uniform generation of some class of objects seems difficult, a fallback position is to investigate approximate solutions. One   approach is to use the Markov Chain Monte Carlo (MCMC) method. In this, an ergodic Markov chain on the set of graphs with given degrees is designed so that the stationary distribution is uniform. Then, the random graph obtained after taking a sufficiently large number of steps (i.e.\ the so-called mixing time of the chain)  has distribution that is close to   uniform. Jerrum and Sinclair~\cite{JS} gave a fully polynomial time approximation scheme (FPTAS) for approximate uniform generation of graphs with given degrees. Their algorithm  works for a large class of degree sequences. In particular, it works for all regular graphs.  The bound on the mixing time, and hence the runtime of the algorithm required for any guarantee of approximation to the uniform distribution, is a polynomial in $n$ (and is not specified, nor optimised, in their paper). Kannan, Tetali and Vempala~\cite{KTV} used another Markov chain, whose transitions are defined by switching two properly chosen edges in a certain way, to approximately sample random regular bipartite graphs. Again, they proved that the mixing time is polynomial in $n$ without specifying an asymptotic bound. This Markov chain was further extended by Cooper, Dyer and Greenhill~\cite{CDG} for generation of random $d$-regular graphs. They showed that the mixing time is then bounded by roughly $d^{24}n^9\log n$. Very recently, Greenhill~\cite{G4} extended that result  to the non-regular case, with a bound on the mixing time of $\Delta^{14}M^{10}\log M$ where $\Delta$ and $M$ denote the maximum and total degrees respectively. This result applies only for $3\le \Delta\le \frac14\sqrt M$.  These MCMC-based algorithms generate  every graph with a probability within a factor $1\pm \eps$ of the probability in the uniform distribution, and $\eps>0$ can be made arbitrarily small by running the chain sufficiently long. So the output of such algorithms is almost as good as that from an exactly uniform sampler, for any practical use.  However, the provable bounds on the degree of the polynomials involved are too high for any practical use. Actually, there is a general belief that MCMC algorithms such as this do  mix much faster, and produce near-uniform results  much more quickly, than what has been proved. If this were proved, they could  give practical algorithms for approximately uniform generation.  Without such a proof, one can only guess the accuracy of the results.

   There is a variation of MCMC called {\em coupling from the past}~\cite{PW}, which is capable of producing a target distribution precisely. However, it is often hard to prove useful bounds on the running time, and the method has not been successfully applied to generating random graphs given degrees.  
 
There are   algorithms   faster than the   MCMC-based ones, that generate graphs with a weaker approximation of the distribution  to the uniform. 
Steger and Wormald~\cite{SW} gave  an $O(d^2n)$-time  algorithm that generates random $d$-regular graphs for $d$ up to $n^{1/28}$, where all graphs are generated with asymptotically the same probability. Kim and Vu~\cite{KV} proved that the same algorithm  works to the same extent for all $d\le n^{1/3-\eps}$. Bayati et al.~\cite{BKS} subsequently modified and extended it for the non-regular case, under certain conditions, and generated random $d$-regular graphs for all $d\le n^{1/2-\eps}$, but with a weaker approximation to uniform (bounding the total variation distance by o(1)).  Using a different approach, Zhao~\cite{Zhao} obtained an $O(dn)$-time approximate algorithm, which can generate $d$-regular graphs for $d=o(n^{1/3})$ under weak approximation (bounding the total variation distance).  These algorithms are much faster than using   MCMC, but the approximation (to the uniform) is achieved only asymptotically as $n\to\infty$. Thus, whereas MCMC permits the approximation error to be made arbitrarily small for graphs of a fixed size by running on the chain sufficiently long, the approximation error in~\cite{BKS,KV, SW,Zhao} depends on $n$ and cannot be improved by more computation.

%Jerrum and Sinclair~\cite{JS} [describe fast uniform generation and the types of near uniformity they looked at];
%McKay and Wormald~\cite{MWgen} showed how to generate random $d$-regular graphs on $n$ vertices with precisely the uniform probability distribution. [Something about time complexity] Since then, Steger and Wormald~\cite{SW} [blah blah].  Then~\cite{KV} BLAH BLAH and  ~\cite{BKS} showed that the algorithm in~\cite{SW} does nearly uniform generation for $d$ up to $n^{1/2-\delta}$.[Must check if they actually modified it; they did so for non-regular, which probably deserves a brief mention here.] (Possibly refer to \cite{KV2}.)

The main purpose of the present paper is to introduce a new general framework of uniform generation of combinatorial structures. We will  apply it in this paper to the special, nevertheless particularly interesting, case of random $d$-regular graphs on $n$ vertices. The result is an algorithm, \NameA, effective for   $d=o(\sqrt n)$. This significantly improves the bounds on $d$ in~\cite{MWgen}.
\NameA\  is an exactly uniform sampler and thus there is no uncontrollable distortion as in~\cite{BKS,KV, SW,Zhao}. Moreover, its expected running time   per graph generated is $O(nd^3)$ (the same as~\cite{MWgen}), which remains quite practical, comparing favourably with the quite impractical running time bounds of MCMC samplers.
\begin{thm}\lab{thm:uniform}
 Algorithm  \NameA\  generates $d$-regular graphs uniformly at random. 
\end{thm}
\begin{thm}\lab{t:complexity}
\NameA\ can be implemented so that for  $1\le d=o(\sqrt{n})$,  the expected time complexity for generating  a graph  is $O(nd^3)$.
\end{thm}
    The same time complexity seems likely to apply under the weaker assumption that $d=O(\sqrt{n})$, as one might expect comparing with the constraint $d=O(n^{1/3})$ in~\cite{MWgen}. However, this is not proved here, the difficulty being that the generation method is different and the analysis is now much more intricate.  
  
   In some applications, one  might  care more for a low time complexity than a perfect uniform sampling. As a byproduct of Theorem~\ref{thm:uniform}, by omitting  a particular feature of \NameA\ that requires excessive computation,  we will obtain  a simpler, linear-time algorithm  (that is, linear in the output size, which is the number of edges)  
 called  \NameB, which approximately generates a random $d$-regular graph.  \NameB\  will be defined in Section~\ref{sec:main}.  The total variation distance between the distribution of the output of  \NameB\  and the uniform distribution is bounded as follows.

\begin{thm}\lab{thm:approx}
Algorithm   \NameB\  randomly generates a $d$-regular graph whose total variation distance from the uniform distribution is $o(1)$, for any $d=o(\sqrt{n})$. Moreover, the expected number of steps required for generating a graph  is  $O(dn)$.
 \end{thm} 
 This improves the running time $O(nd^2)$ of~\cite{BKS}, and the range $d=o(n^{1/3})$ of~\cite{Zhao}, while achieving a  quality of output distribution similar to both. A similar modification to the algorithm of~\cite{MWgen} would achieve the same result, provided that $d=O(n^{1/3})$.

We outline our general framework in Section~\ref{sec:framework}. The framework includes several operations and parameters that will be defined in accordance with the types of combinatorial structures to be generated.  For the application in the present paper, this framework is used in each of the three phases of the main algorithm, called \NameA, that is a uniform sampler for $d$-regular graphs on $n$ vertices.  The framework requires   operations and parameters to be defined
 for the various phases of \NameA. These are given in Sections~\ref{sec:switchings} and~\ref{sec:rho2} and at the beginning of Sections~\ref{sec:UBLB},~\ref{loopreduction} and~\ref{triplereduction}.   The  full structure of this paper is explained at the end of the following section, after the new ideas have been discussed in relation to the algorithm {\em DEG} of~\cite{MWgen}.

 We note that several papers  
 have      adapted the   approach of~\cite{MWgen} for generation  of other structures (e.g.\ McKay and Wormald~\cite{genlatinrect}),   and also for enumeration (e.g.\ Greenhill and McKay~\cite{GM}). Such works have not led to   improvements in   the result~\cite{MWgen} achieves for $d$-regular graphs. We expect the ideas introduced here will filter out to improved results for several kinds of structures, including graphs with non-regular   degree sequences. Such issues will be  examined  elsewhere.

\section{The old and the new}
\lab{MW}
  
In this section we first summarise the procedure {\em DEG} used  in~\cite{MWgen}, which provides some of the foundations required for applying our method to regular graphs. We then highlight the new ideas used in \NameA, and give a skeleton description of that algorithm. Finally, we describe the layout of the paper in relation to exposing the new framework and defining and analysing   \NameA\ and \NameB.

For generating random  graphs with given degrees, we use the {\em pairing model}, first introduced in~\cite{B}, defined as follows.  Let  ${\bf d} = (d_1,\ldots,d_n)$ be a degree sequence (thus $\sum_{i=1}^n d_i$ is always assumed to be even). 
% The {\em configuration model}, also called the {\em pairing model}, is defined as follows. 
Represent each vertex $i\in [n]$ as a bin $v_i$. Place $d_i$ distinct points in bin $v_i$ for every $1\le i\le n$. Take a uniformly random perfect matching of the $\sum_i d_i$ points. This perfect matching   is called a {\em pairing}; each pair of points joined in the matching is called a {\em pair} of the pairing. Note that each pairing $P$ corresponds to a multigraph with degree sequence ${\bf d}$, denoted by $G(P)$, obtained by regarding each pair in the pairing as an edge. Moreover, by a simple counting argument, we see that every simple graph of degree sequence {\bf d} corresponds to the same number of pairings. Thus, letting $\Phi$ denote the whole set of pairings and   $\simple\subseteq \Phi$  the set of pairings corresponding to simple graphs, if an algorithm can   generate   a pairing $P\in \simple$ uniformly at random, then $G(P)$ has the uniform distribution over all graphs with degree sequence {\bf d}. 

 Given a pairing $P$, and two  vertices $i$ and $j$, the set of pairs between $i$ and $j$ in $P$, if non-empty, is called an {\em edge}, i.e.\ edge $ij$, and the size of the set is the {\em multiplicity} of $ij$. If  the multiplicity of $ij$ is 1, then $ij$ is a {\em single} edge; otherwise it is a {\em multi-edge}. In particular, we say $ij$ is a {\em double edge},   a {\em triple edge}, or a {\em quadruple edge}, if  its multiplicity is two, three, or four respectively. An edge $ij$ with $i=j$ is called a {\em loop} at $i$.  
  \medskip

 \no {\bf  Outline of {\em DEG}}
 \smallskip

    The   algorithm {\em DEG} begins by generating a uniformly random pairing $P_0\in \Phi$. An appropriate set $\acceptableMW\subseteq \Phi$ is pre-defined, such that pairings in $\acceptableMW$ have no multi-edges other than  double non-loop edges, and have   limited numbers of double edges and loops.    If  $P_0\notin \acceptableMW$, then the algorithm terminates.  We call this an {\em initial rejection}. The set  $\acceptableMW$ is chosen in such a way that $\pr(P_0\in \acceptableMW)$ is bounded away from 0. This is possible when $d=O(n^{1/3})$.
If no initial rejection occurs, two phases are applied in turn. In the first phase, called {\em loop  reduction},   $P_0$ is used to obtain a uniformly random pairing $P_1\in \acceptableMW$ with no loops  but the same number of double edges as $P_0$. Then, starting with $P_1$, the algorithm enters the  {\em double-edge reduction} phase, and obtains a uniformly random pairing $P_2\in\simple$. (Recall that  $\simple$ denotes the set of  pairings   associated with simple graphs.) Thus, in each phase, the number of undesirable structures (which are in turn loops and double edges) is reduced to 0. Termination is also possible during the phases, due to rejection aimed at maintaining uniformity. 
In that case, no pairing is outputted by the algorithm. Otherwise, the pairing outputted by the last phase is the output of the algorithm. Naturally, the algorithm is repeated until output occurs.

For exposition purposes, we focus on the second phase, which starts from a pairing $P_1\in \acceptableMW$. In particular, $P_1$ has   no loops and at most $\imax$ double edges, where $\imax$ is specified in the definition of $\acceptableMW$.   Then $P_1$ is distributed uniformly at random (u.a.r.) in $\bigcup_{0\le i\le \imax}\state_i$, where   $\state_i$ denotes the set of   pairings in $\acceptableMW$ containing exactly $i$ double edges. (In the present paper, these sets are called {\em \groups}.)  We will describe how, with probability bounded away from 0, we can use  $P_1$ to u.a.r.\ generate a pairing in $\state_0$ (note that $\state_0=\simple$ in the second phase).

Initially, the algorithm sets $P=P_1$.  There is an inductive step  which  assumes that, conditional on    $P\in \state_i$, $P$ occurs  uniformly at random  in $\state_i$.  The algorithm then randomly performs a certain kind of operation called a {\em switching}, that produces a pairing $P'\in \state_{i-1}$, and then $P$ is reset equal to $P'$. (The definition of the switching is perhaps not necessary at this point, but the curious reader may consult   Figure~\ref{f:doubleI}, Section~\ref{sec:switchings}.)
The general idea is to reject $P'$ with a small probability, which is a function of $P$ and $P'$, so that $P'$ becomes a uniformly random member of $\state_{i-1}$ if not rejected. 
This process is iterated until reaching some  $P \in\state_0$ which is then outputted  as $P_2$ by {\em DEG}. By induction, $P_2$ is uniformly distributed over $\state_0$.

We next consider the probability of rejection, which is crucial.
Switchings are defined in such a way that the number $f(P)$ of switchings that can be performed on $P$ depends only weakly on anything other than  how many double edges $P$  contains. A parameter $m_f(i)$ is specified such that
$$m_f(i)\ge \max_{P\in \state_i}f(P).$$
In discussing the inductive step, we condition on the event $P\in \state_i$. Firstly one of the $f(P)$ switchings is chosen u.a.r., and with probability
$f(P)/m_f(i)$
the switching is performed, to obtain $P'\in   \state_{i-1}$. Otherwise, with the remaining probability, the algorithm terminates. We call this termination an {\em f-rejection} (where `f' stands for `forward'). Since $P$ is uniformly distributed in $\state_i$, the probability that $P'\in \state_{i-1}$ is generated at this point by the switching  is $|\state_i|^{-1} m_f^{-1}b(P')$, where $b(P')$ is the number of switchings that lead to $P'$ from pairings in $\state_i$. At this point, the algorithm accepts $P'$ with probability $ m_b(i-1)/b(P')$ where $m_b(i-1)$ is a pre-determined parameter satisfying
$$m_b(i-1)\le\min_{P''\in \state_{i-1}}b(P'').$$
If     $P'$ is not accepted, the algorithm terminates, and this is called a {\em b-rejection} (where `b' stands for `backward').
The probability that a given pairing $P'\in \state_{i-1}$ was produced is now
$|\state_i|^{-1} m_f(i)^{-1}m_b(i-1)$, which does not depend on $P'\in \state_{i-1}$. Thus, if it reaches this point, the algorithm has generated a uniformly random member of $\state_{i-1}$, and the inductive step is finished.

%We return to consider the random $P\in \bigcup_{i\ge 0}\C_i$. The algorithm terminates immediately if $P$ has more than $\imax$ double edges. This can be accomplished by an appropriate choice at the initial step, so we call it an {\em initial rejection}. Here $\imax$ is predetermined to ensure that the probability of initial rejection   is small. In general, we may choose $\imax$ equal to the say twice the expected number of double edges; then Markov's inequality guarantees the probability of \nke{an initial} rejection is at most $1/2$. For the $d$-regular case, this expected number is   $\Theta(d^2)$.  If there is no rejection, $P$ is then in a set $\C_i$ as described above, and the iterative step is applied repeatedly until $P\in \C_0$ or rejection occurs. In the former case, it gives a   graph u.a.r.\ as desired.

The range of applicable degree sequences for {\em DEG} is determined by the probability of rejection at some time during the algorithm. In~\cite{MWgen}, the bound $d=O(n^{1/3})$ for $d$-regular graphs could not be weakened because the probability of rejection would get too close to 1. This is caused by the probability of f-rejection becoming too large, due to the increasing gap between the typical value of $f(\P)$ and its maximum, $m_f$.

The first phase, loop reduction, is similar, except that of course a different switching is used. There are less loops 
than double edges in expectation, so the crucial phase to improve, in order to relax the upper bound on $d$, is the second phase.

\medskip

\no {\bf New features in \NameA}
\ss

 Our new approach extends that used in {\em DEG}, introducing some major new  features in both the algorithm specification  and its analysis, employed specifically in the double-edge reduction phase. For one thing, we narrow the gap between $m_f$ and the average value of $f(P)$  by permitting certain switchings,  called class B,  that  do not have the desired effect on the number of double edges. The total number of permitted switchings is then less dependent on $P$, and $f(P)$ increases on average. The result is a lower probability of   f-rejection. The other new feature of the algorithm, which we call {\em boosting}, raises the value of $m_b$ by occasionally performing a different type of switching that targets the creation of some otherwise under-represented elements of a set $\C_i$.  This  reduces the probability of having a b-rejection. These changes necessitate  several associated alterations to the algorithm. The most notable alterations are: (i) the algorithm no longer  proceeds through the sets   $\C_i$ step-by-step, decreasing $i$ by 1 at each step, though this is still the most common type of step; (ii) unlike in~\cite{MWgen}, the probability that a pairing is reached in the algorithm (at all, or at a given step) no longer depends only on the set $\C_i$  to which it belongs; (iii) not all switchings to pairings in such a set  will be performed with the same probability. As a result of these changes, the analysis of the algorithm is entirely different.   In particular, we are forced to relinquish maintaining the property that, at each step, conditional upon the current pairing being within the set $\C_i$, it is distributed uniformly in that set (except for the case $i=0$).  Instead, we focus on the  expected numbers of visits to the states in the associated Markov chain. 

%%%%%
%%%%%
\medskip

\no {\bf The rest of the story}
\ss 

As in~\cite{MWgen},   a set $\acceptable\subseteq \Phi$  is specified such that $\pr(\acceptable)$ is bounded away from zero. Here, $\gamma>0$ is a pre-determined constant. It can be altered for better performance of the algorithm for particular $n$ and $d$ using the results of this paper.  A uniformly random pairing $P\in\Phi$ is generated, and   initial rejection is performed if   $P\notin\acceptable$. Pairings in $\acceptable$ will in general contain loops, and double and  triple non-loop edges, but no other multi-edges. Three phases are performed  sequentially, for reduction of loops,   then  triple edges,  and finally  double edges.

Since the new  features of \NameA\   will be useful in other contexts, we   set up a general framework for the description of a phase in Section~\ref{sec:framework}. This new framework results in a different analysis from~\cite{MWgen}, which will be given in Section~\ref{s:phase}, with a glossary provided as Section~\ref{s:glos}. The proof that each phase ends with a uniformly random object of the required type is rather involved, so an example appropriate to the double-edge reduction phase is given in Section~\ref{sec:rho}. This includes an illustration of how to set some of the parameters of a phase appropriately. Since the only nontrivial phase in \NameA\  is for double-edge reduction, the definitions of the switchings and other parameters  in this phase, and the analysis required for bounding the time complexity, is done in Section~\ref{sec:double}.

With this out of the way, the basic anatomy of \NameA\ is completed in Sections~\ref{sec:A}  and     
   Section~\ref{sec:loop-triple}.
In Section~\ref{sec:A}, we  specify the set $\acceptable$ of pairings that do not trigger initial rejection, and bound the probability of an initial rejection. Then, in Section~\ref{sec:loop-triple}, we define the first two phases, for reductions of  loops and of triple edges respectively. These two phases are much simpler than phase 3. The switchings employed in these two phases are the same as used in~\cite{MWgen}, and are defined in Section~\ref{sec:loop-triple}.  The analysis is similar to that in~\cite{MWgen}, though we now have a higher upper bound on degree. The probability of a rejection occurring in these phases is also bounded.

Finally in Section~\ref{sec:main}, we prove the main theorem by bounding the expected running time. This is basically determined by the task of  computing the probability of rejection, in particular $b(P)$. By ignoring rejections (carrying on regardless), we will obtain the approximate sampler \NameB\ in Section~\ref{sec:main}. We prove that \NameB\ runs in linear time in expectation for generating one random regular graph, and   bound the total variation distance between the distribution of the output of \NameB\ and the uniform distribution.

     %%%%%%%%%%%
\section{General description of  a phase}
\lab{sec:framework}

%\ncw{I've just realised that this section has changed. It's no longer a general description of the (new part of  the) method. My original intention when writing what is in the phase description was to stick to showing what the new things are, not start out describing initial rejections, three phases etc ... having all that here makes the whole thing rather messy. 

%WE could transfer those things to the second half of section 2, or blend them in.  This material has somehting in common with hte start of section 7, where we get back to our specific example. I'd like to think about just using this material in SEction 2 as part of the description of DEG.}

  We present here the definition of a phase  that will be common to any application of our approach to reducing the number of occurrences of an undesired configuration  (for instance, a double edge), by using repetitions of operations, called switchings here, that can be defined to suit the application.

A phase begins with a set $\Phi_1$  partitioned into sets $\bigcup_{0\le i\le\imax}\state_i$, called {\em \groups}, for some integer $\imax\ge 0$.    For each $P\in\Phi_1$, set $\state(P)= i$ if $\state_i$ contains $P$. A set of  possible operations, called switchings, is specified. Each switching converts some   $P\in \Phi_1$ to another element of $\Phi_1$. Each switching has both a   {\em type}   and a  {\em class}. The set of possible types is denoted $\types$, and the set of possible classes is denoted $\classes$. The phase begins with a random element $P\in \Phi_1$ with a known distribution $\pi_0$, and either outputs a uniformly random element of $\C_0$, or terminates with no output (rejection). The parameters specified for the phase   are 
 $ \rho_\tau(i)$,  $\UB_\tau(i)$ and $\LB_\alpha(i)$,  for each switching type $\tau$ and class $\alpha$, and each $i\le \imax$.

The parameters  $ \rho_\tau(i)$ satisfy $\sum_\tau  \rho_\tau(i)\le 1$ for each $i$.   If an element $P\in\state_i$ ($i>0$) arises during the phase, a switching type $\tau$ is chosen with probability $\rho_{\tau}(i)$.\label{def:rho} If $\sum_{\tau}\rho_{\tau}(i)<1$, a rejection will occur with the remaining probability; we call this a t-rejection (where `t' stands for `type').

  Given a switching type $\tau$ and an element $P$, the  number of switchings of type $\tau$ that can   be applied to   $P$ is denoted by $f_\tau(P)$\label{def:f}.  The parameters $\UB_\tau(i)$ satisfy 
  $$
\UB_\tau(i)\ge \max_{P\in \C_i}f_\tau(P).\label{def:UB}
$$ 
Similarly,  for a switching class $\alpha$, the  number of switchings of class $\alpha$ that can be applied to other elements to produce  $P$ is denoted by $b_\alpha(P)$\label{def:b}.
The parameters $\LB_\alpha(i)$ satisfy 
$$
\LB_\alpha(i)\le\min_{P\in \C_i}b_\alpha(P).\label{def:LB}
$$ 

The phase consists of repetitions of a {\em switching step}, specified as follows.
\bigskip

 Given $P\in \C_i$,  
 \begin{description} 
 \item{(i)} If $i=0$, output $P$.
 \item{(ii)}  Choose a type: choose $\tau$ with probability $\rho_\tau(i)$, and with the remaining probability, $1-\sum_\tau \rho_\tau(i)$,  perform a t-rejection. Then select u.a.r.\ one of the type $\tau$ switchings that can be performed on $P$.
\item{(iii)} Let   $P'$ be the element that the selected switching would produce if applied to $P$, let $\alpha$ be the class of the selected switching   and let   $i'=\state(P')$. Perform an f-rejection with probability $1-f_{\tau}(P)/\UB_{\tau}(i)$ and then perform a b-rejection with probability $1-\LB_{\alpha}(i')/b_{\alpha}(P')$;
\item{(iv)} if no rejection occurred, replace $P$ with $P'$.
\end{description}
The switching step is repeated until the phase terminates, which happens whenever  an element  $P\in \state_0$ is reached or a rejection occurs.  
\smallskip

 Note that in each switching step only a switching type is selected, not a switching class. Only after a switching is chosen in (ii) is the class of the switching determined. 
 
  To complete the definition of a phase, it is sufficient to specify $\Phi_1$, $\imax$, the sets $\state_i$, the set of switchings  and their types and classes, and the numbers $\rho_\tau (i)$, $\UB_\tau(i)$ and $\LB_\tau(i)$.
These parameters will be carefully chosen in such a way that the expected number of times that a given element in $\state_i$ is visited during the phase depends only on $i$. 
Given this, and the fact that termination of the phase occurs as soon as $\state_0$ is reached, it follows that the element outputted is distributed u.a.r.\ from $\state_0$.  
Subject to this, in choosing the parameters we also aim to keep the probability of rejection small.

%%%%%%%%%%%%%%%%%%%%%%%%%%%%%%%%%%%%%%%%%%%%%%%%%%%%%%%%%%%%%%%

\section{Glossary}
\lab{s:glos}

For the benefit of the notation-weary reader, we list some terms already defined, and  some soon to be defined.

\bigskip

\noindent
For any element $P$:

$\sigma(P)$ is the expected number of times the algorithm passes through $P$ (defined in Page~\pageref{def:sigma}). 

$f_\tau(P)$ is number of ways that a switching of type $\tau$ can be performed on $P$ (defined in Page~\pageref{def:f}).

$b_\alpha(P)$ is number of  switchings of class $\alpha$ that produce $P$ (defined in Page~\pageref{def:b}).

$P_t$ is the element obtained after step $t$, if   no rejection has occurred. 

$q_{P,P'}$ is the transition probability from $P$ to $P'$.

%$\state^+_{\tau,\alpha}(P)$ is  the set of elements that can be obtained from $P$ by a type $\tau$,   class $\alpha$ switching. 

$\state^-_{\tau,\alpha}(P)$: the set of elements that can produce $P$ using a type $\tau$, class $\alpha$ switching (defined in Page~\pageref{def:state-}).

\bigskip

\noindent
For $P\in \C_i$:
 
$\rho_\tau(i)$ is probability of choosing switching type $\tau$ to apply to $P\in\state_i$ (first appearing in Page~\pageref{def:rho}). 

$\sigma(i)$ is such that $\sigma(P)=\sigma(i)$ for all $P\in\C_i$ under constraints to be enforced (first appearing in~\eqn{equalise}).

$\C(P)$ is another name for $i$. (That is,  the index of the \group\ containing $P$.)

\bigskip

\noindent
Special states  in the Markov chain:

$\reject$: a state denoting that rejection has occurred.

$\output$: each element in $\state_0$   has transition probability 1 to $\output$. \bigskip

\noindent In a phase:

$\Phi_1$: the set of all possible elements arising in the phase.

$\imax$: the maximum integer $i$ such that $\state_i\in\Phi_1$.

$\Tau$: the set of switching types.

$\Alpha$: the set of switching classes.

$\UB_\tau(i)$: pre-determined parameters satisfying $f_{\tau}(P)\le \UB_\tau(i)$ for every $P\in \state_i$ (first appearing in Page~\pageref{def:UB}).

$\LB_\alpha(i)$: pre-determined parameters satisfying $b_{\alpha}(P)\ge \LB_\alpha(i)$ for every $P\in \state_i$ (first appearing in Page~\pageref{def:LB}).

\section{General analysis of a phase}
  \lab{s:phase}
  
%\ncu{It was confusing having a section with a general framework and a particular example that then described what the rest of the paper did (if the merging you suggested had taken place. Try this layout.}
   
In this section we lay the groundwork for specification of the predefined  parameters described in Section~\ref{sec:framework}, in such a way that the algorithm performs the desired uniform sampling.   Consider a phase with specified switchings and parameters $\C_i$,  $\rho_\tau (i)$, $\UB_\tau(i)$ and $\LB_{ \alpha }(i)$ for all appropriate $i$ and $\tau$, as well as a specified value of $\imax$.   Assume that 
  % parameters $\rho_\tau$, $\UB_\tau(i)$ and $\LB_{\alpha}(i)$ satisfy
  \bea
 \rho_{\tau}(i)\ge 0 && \mbox{for all $0\le i\le \imax$ and all $\tau\in \Tau$};\lab{ParamCond1}\\
  \sum_{\tau\in \Tau}\rho_{\tau}(i)\le 1 && \mbox{for all $0\le i\le \imax$};\lab{ParamCond2}\\
 f_{\tau}(P)\le \UB_\tau(i)  && \mbox{for all $0\le i\le\imax$, $P\in \state_i$ and for each $\tau\in\Tau$};\lab{ParamCond3}\\
 b_{\alpha}(P)\ge \LB_{\alpha}(i) && \mbox{for all $0\le i\le\imax$, $P\in \state_i$ and for each $\alpha\in\Alpha$}. \lab{ParamCond4}
  \eea
Recall that $\Phi_1=\bigcup_{i=0}^{\imax} \C_i$. 
 The parameters determine a Markov chain, denoted by $\markov$,  on states $\Phi_1\cup \{\reject,\output\}$, where $\reject$ and $\output$ are two artificially introduced absorbing states. The chain moves from  an element $P$ directly to $\reject$  if rejection occurs  in a switching step from $P$ to $P'$  (rather than to $P'$ being considered at the time of rejection), and  it moves to $\output$ with probability 1 in the next step after reaching any element in $\C_0$.  We refer to the state at step $t$ as $P_t$, permitting $P_t=\reject$ or $\output$.

We make the following assumption about $\markov$:
\bean
&&\mbox{(A1) all states in $\Phi_1$ are transient in $\markov$.} 
% &&\mbox{(A2) for $P,P'\in \Phi_1$, at most one switching of a given type converts $P$ to $P'$;}\\
% &&\mbox{(A3) the initial probability distribution is $1/|\Phi_1|$ at each $P\in\Phi_1$, and zero elsewhere.}
\eean 
Let $Q$ be the matrix of transition probabilities between all  states of $\markov$ in $\Phi_1$. Thus  $Q=(q_{P,P'})$  where $q_{P,P'}$   is the transition probability from $P$ to $P'$. 

The transition probability  $q_{P,P'}$  can be computed as follows. Assume $P\in \state_i$ and  that $S$ is a type $\tau$, class $\alpha$ switching that converts $P$ into $P'\in\state_j$. 
Condition  on state $P$ being reached in $\markov$. Then, from part (ii) of the switching step,  the probability that  $S$ is chosen equals $\rho_{\tau}(i)/f_\tau(P)$.
   Hence, from part (iii), the probability $S$ is performed and   neither t- nor f-rejection  occurs  is $\rho_{\tau}(i)/\UB_\tau(i)$. On the other hand,
the probability that b-rejection does not occur is $\LB_{\alpha}(j)/b_{\alpha}(P')$.   Hence,
\be\lab{transitionprob}
q_{P,P'} =\sum_{(\tau,\alpha)}   
N_{\tau,\alpha}(P,P') \frac{\rho_{\tau}(i)\, \LB_{\alpha}(j)}{\UB_{\tau}(i)\, b_\alpha(P')},
 \ee
where $N_{\tau,\alpha}(P,P')$ is the number of switchings of type $\tau$ and class $\alpha$ that convert  $P$ to $P'$.

By assumption (A1), $Q$ is the submatrix of the transition matrix that refers to the transient states. Hence, the matrix  $(I-Q)^{-1}$ exists. Indeed, this matrix is known as the fundamental matrix of $\markov$, and it is clearly equal to $I+Q+Q^2+\cdots$. (An easy argument shows that this series is  convergent because these states are transient.) 
Moreover, the $(P,P')$ entry of $(I-Q)^{-1}$ is clearly the expected number of visits to state $P'$ given that the chain starts in state $P$ (where being in the initial state is counted as a visit). Given a (row) vector
${\vec \pi_0}$ for the initial distribution of the transient states,  define  $\sigma(P)$~\label{def:sigma} to be the expected number of times that $P$ is visited in $\markov$. Then  the   vector $\vec \sigma$, composed of the values $\sigma(P)$, is given by  
\bel{visits}
 {\vec \pi_0}(I-Q)^{-1}=\vs.
\ee
  A key feature of our approach is to specify $\rho_{\tau}(i)$ in such a way that $\sigma(P)$ depends only on $\state(P)$, i.e., there are fixed numbers $\sigma(i)$ ($0\le i\le \imax$) such that
\bel{equalise}
\mbox{for all $i$, and all $P\in \state_i$,}\quad \sigma(P)=\sigma(i).
\ee
To aid in finding such $\rho_\tau$ easily, we require that, for a given switching $S$, the expected number of switching steps  during the phase in which $S$ is chosen in (ii) for which f-rejection does not occur in (iii) is some number    $q_\alpha(j)$ depending only on  the class $\alpha$ of $S$ and the set $\C_j$ containing the element it creates. Considering the derivation of~\eqn{transitionprob}, this is equivalent to requiring that, for all $j$ and $\alpha$,
\be\lab{equaliseclass}
    \frac{\sigma(i)\rho_{\tau}(i) }{\UB_{\tau}(i) }=  q_\alpha(j) \ \mbox{for all $(i,j,\tau,\alpha)\in \quads$}
 \ee
where $\quads$ is the set of all $(i,j,\tau,\alpha)$ for which there exists a switching of type $\tau$ and class $\alpha$ taking an element in $\state_i$ to an element in $\state_j$.

Rewrite~\eqn{visits} as ${\vec \sigma}= {\vec \pi_0} + {\vec \sigma} Q$, and note from~\eqn{transitionprob} that the component of ${\vec \sigma} Q $ referring to (i.e.\ indexed by) $P'\in\state_j$ is
$$
 \sum_P \sigma(P)q_{P,P'} =
   \sum_{(\tau,\alpha,P)}  \sigma(P)  N_{\tau,\alpha}(P,P') 
  \frac{\rho_{\tau}(\state(P))\, \LB_{\alpha}(j)}{\UB_{\tau}(\state(P))\, b_\alpha(P')},
$$
where the sum is over all $\tau$, $\alpha$ and $P$ for which there is at least one type $\tau$, class $\alpha$ switching  that converts $P$ into $P'$. By~\eqn{equalise} and~\eqn{equaliseclass}, this summation is
$$
 \sum_{(\tau,\alpha,P)}  N_{\tau,\alpha}(P,P')  \frac{q_\alpha(j) \LB_{\alpha}(j)}{  b_\alpha(P')} 
 = \sum_{ \alpha}   q_\alpha(j) \LB_{\alpha}(j) 
$$
since   for each $\alpha\in\Alpha$, the number of class $\alpha$ switchings that converts some element $P$ to $P'$ is $\sum_{\tau,P}N_{\tau,\alpha}(P,P')$, which is $ b_\alpha(P')$ by definition. 
Thus, provided that~\eqn{equalise} and~\eqn{equaliseclass} hold,~\eqn{visits} is equivalent to
\bel{sigmasGen}
\sigma(P)= {\vec \pi_0}(P) + \sum_{ \alpha\in\Alpha}   q_\alpha(j) \LB_{\alpha}(j) \ \mbox{for all $j$ and $P\in\state_j$}.
\ee 
Noting that~\eqn{visits} determines $\vec \sigma$, we have proved the following lemma.
%%%%%%%%%%%%%%%%%%
\begin{lemma} \lab{l:equalised}
Suppose that for given numbers $\rho_\tau(i)$,  $\LB_{\alpha}(j)$ and $\UB_{\tau}(j)$ ($i,j\in[0, \imax]$, $\tau\in\Tau$, $\alpha\in \Alpha$)  satisfying the conditions~\eqn{ParamCond1}--\eqn{ParamCond4}, there is a simultaneous solution  $(\sigma(i), q_{\alpha}(i))_{0\le i\le \imax}$ to equations~\eqn{equaliseclass} and~\eqn{sigmasGen}. Then, for each $i\in [0,\imax]$, the expected number of visits to any given element in $\state_i$ is $\sigma(i)$. \end{lemma}
%%%%%%%%%%%%%%%%%% 
Henceforth in this paper, we assume that the initial distribution is uniform over $\Phi_1$, that is, ${\vec \pi_0}(j)=|\Phi_1|^{-1}$ for all $j\in \Phi_1$.

 \no {\bf Remark.}  Consider the case that there is exactly one type $\tau$ and one class $\alpha$ of switchings involved in a phase, and  each switching converts an element in $\state_{i+1}$ to another element in $\state_{i}$ for some $i\ge 0$. Then  we may simply set $\rho_{\tau}(i)=1$ for every $0\le i\le \imax$. Clearly (A1) is satisfied, as there is no cycling in the Markov chain, and~\eqn{equaliseclass} and~\eqn{sigmasGen} combine to give
\[
\sigma(j)=\frac{1}{|\Phi_1|}+\frac{\sigma(j+1)}{\UB_{\tau}(j+1)}\cdot \LB_{\alpha}(j).
\]
Uniformity is guaranteed as $\sigma(j)$ can be inductively computed from $\sigma(j+1)$. This inductive approach is exactly the  essence of {\em DEG} described in  Section~\ref{MW}, and hence the method in~\cite{MWgen} is a special case of our general method, obtained by setting $\rho_{\tau}(i)=1$. It is the possibility of using different types and classes of switchings, the flexibility of setting non-trivial values to $\rho_{\tau}(i)$, and the much more flexible choice of Markov chains permitting cycling, that provides the power of the approach in the present paper. 
 
%%%%%%%%%%%%%%%%%%%%%%%%%%%%%%%%%%%%
  \section{An example: calculating $\rho$ and proving uniformity}
\lab{sec:rho}

The parameters $\UB_{\tau}(i)$ and $\LB_{\alpha}(i)$ will be  specified, depending on the particular application, such that~\eqn{ParamCond3} and~\eqn{ParamCond4} are satisfied. The tightness of these bounds on $\max f_{\tau}(P)$ and $\min b_{\alpha}(P)$ effectively influence the efficiency of the phase, as tighter bounds yield smaller rejection probabilities in substep (iii) of a switching step. However, a major task of the design of the phase is to set $\rho_{\tau}(i)$ properly to ensure~\eqn{equaliseclass}, as well as to minimize the probability of a t-rejection. We achieve this by deducing a system of equations and inequalities that the parameters $\rho_{\tau}(i)$ and the variables $\sigma(j)$ must satisfy. Then, we find a desirable solution to the system, bearing in mind the rejection probabilities, and set the value of $\rho_{\tau}(i)$ accordingly. We illustrate this by considering a particular example, developing the analysis in the previous section, which we can use later since it applies to phase 3 of \NameA, where double edges are eliminated.

 We assume that $\Phi_1=\cup_{i=0}^{\imax} \state_i$ and all \groups\ have been specified, as well as parameters $\UB_{\tau}(i)$ and $\LB_{\alpha}(i)$ satisfying~\eqn{ParamCond3} and~\eqn{ParamCond4}. We assume that $\Tau=\{I,II\}$ and $\Alpha=\{A,B\}$  and there will be three kinds of switchings in the phase: 
 IA (type I, class A), converting an element in $\state_{j+1}$ to an element in $\state_{j}$; IB (type I, class B), maintaining in the same \group; and IIB (type II, class B), converting from $\state_{j-1}$ to $\state_j$.  See Figure~\ref{f:chain} for an illustration of the possible transitions into $\C_j$.  
    \begin{figure}[htb]
\hbox{\centerline{\includegraphics[width=7cm]{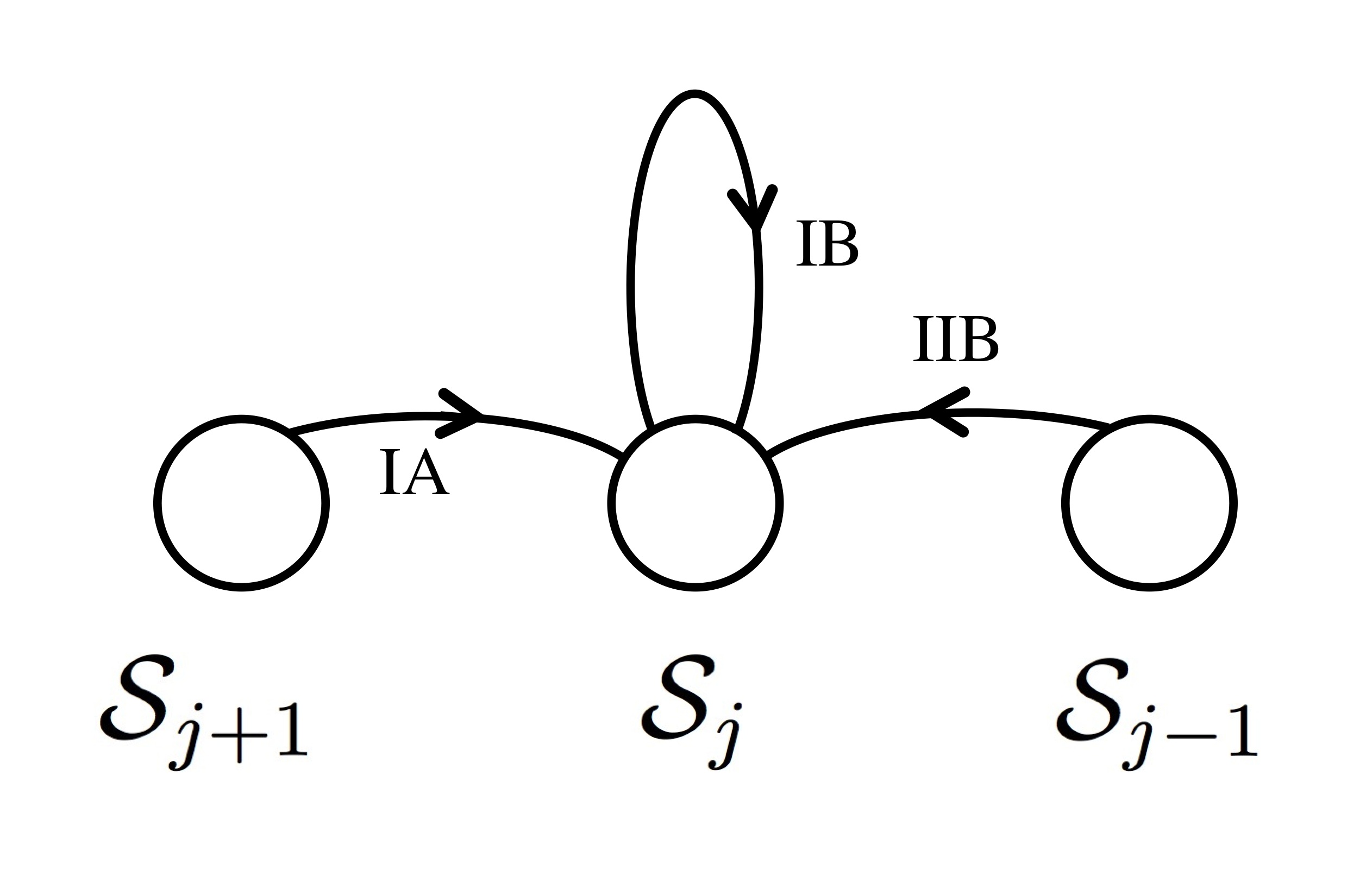}}}
\caption{\it  Transitions into $\C_j$}
\lab{f:chain}
\end{figure}
If $j$ is such that  $\C_{j-1}$ or $\C_{j+1}$  does not exist   in $\Phi_1$, the corresponding \group\ is omitted from the diagram. Additionally, there is no IIB switching  from $\state_0$ to $\state_1$,  and the Markov chain always transits from an element in $\state_0$ to $\output$ in the next step.

We have no need to define $\state_j$ here, nor the  switchings involved. What they are in the case of phase 3 of \NameA\ will be revealed in Section~\ref{sec:double}.

From Figure~\ref{f:chain}, 
we   see that~\eqn{equaliseclass} with $\alpha=$ B in the   two cases  $i=j$ and $i=j-1$ implies that the expected number of times that  a given Class B switching produces $P \in\state_j$ (including the time it is    b-rejected, in part (iii) of the switching step, if that occurs) is
 \bel{equaliseclass1}
 q_{B}(j) =\frac{\sigma(j)\rho_I(j)}{\UB_I(j)} = \frac{\sigma(j-1)\rho_{II}(j-1)}{\UB_{II}(j-1)}\qquad (1\le j \le \imax),
\ee
and at times we will use either of these quantities in place of $q_{B}(j)$.
For similar reasons,   
\bel{qA}
q_A(j) = 
 \sigma(j+1)\rho_I(j+1)/\UB_I(j+1)  \quad  ( 0\le j\le \imax-1).
 \ee

It is convenient to set
\bel{sigma-x}
x_j=\sigma(j)|\Phi_1|,
\ee 
whence, recalling that ${\vec \pi_0}(j)=|\Phi_1|^{-1}$, substitution of~\eqn{equaliseclass1} and~\eqn{qA} into~\eqn{sigmasGen} gives
\bel{recx}
x_j=1+\frac{x_{j+1}\rho_I(j+1)}{\UB_I(j+1)}\cdot \LB_{A}(j)+\frac{x_j\rho_I(j)}{\LB_I(j)}\cdot  \LB_{B}(j) \quad \mbox{($1\le j\le \imax-1$)}.
\ee
We need separate equations for $j\in\{0,\imax\}$ as~\eqn{equaliseclass1} does not hold for $j=0$ and~\eqn{qA} does not hold for $j=\imax$. No $P\in\state_0$ can be reached by any class B switching, because whenever any element in $\state_0$ is reached, the Markov chain proceeds to state $\output$ in the very next step. Therefore, $q_B(0)=0$.
Similarly, no element $P\in \state_{\imax}$ can be reached via a type A switching because $\Phi_1$ does not contain $\state_{\imax+1}$. Therefore, $q_A(\imax)=0$. So  for $j=\imax$ or $0$ we have in place of~\eqn{recx}
provided $i_1>0$
\bel{boundary}
x_{\imax} = 1+\frac{x_{\imax}\rho_I(\imax)}{\UB_I(\imax)}\cdot\LB_{B}(\imax), \qquad
x_0 = 1+\frac{x_1\rho_I(1)}{\UB_I(1)}\cdot\LB_{A}(0).
\ee
 Moreover, the second equality in~\eqn{equaliseclass1} implies 
 \bel{recrho}
 \rho_{II}(j)=\rho_I(j+1)\frac{x_{j+1}}{x_j}\cdot \frac{\UB_{II}(j)}{\UB_{I}(j+1)} \quad \mbox{for all $0\le j\le\imax-1$.}
 \ee
As boundary cases, we require 
\bel{initial}
\rho_I(0)=\rho_{II}(0)=0, \quad \rho_{II}(\imax)=0,
\ee 
where the first two equalities are required because  every element in $\state_0$ is forced to transit to $\output$ once it is reached, and the last  because $\Phi_1$ does not contain $\state_{\imax+1}$. 
%\ncv{You remarked on my comments on the system  - about unique solution ... I agree it's not unique (I hadn't thought about the system for a while) .. of course for different choices of $\rho$'s we get different solutions. But there is an important issue here which I tried to point out from the start and is now {\em completely} hidden in the present writeup here. I was trying to bring it back but just misremembered how it works.  I think it's very nice for the reader to be told what kind of system we might really solve in practice, what will be determined from what etc. This is a bit tricky in the general case but the following suffices for my concerns. OK? (Note that the system solved in section 9 involves only \eqn{recx}--\eqn{recrho})}
 Equations~\eqn{recx}--\eqn{initial} determine a system that   will be required to have a solution. In general, the system will be underdetermined, and this freedom can be used for the convenience of specifying some values of the $\rho_\tau$ variables, to have values suitable for our purpose (in which we aim to keep the probability of rejection to a minimum), and then solving for the remaining variables. 
However, we will also require that    
\bel{atmost1}
\rho_{I}(i)+\rho_{II}(i)\le 1,\quad  \rho_I(i),\ \rho_{II}(i)\ge 0\quad \mbox{for all}\ 0\le i\le \imax.
 \ee 
 This condition ensures that $\rho_{\tau}(i)$ satisfy~\eqn{ParamCond1} and~\eqn{ParamCond2}, as required if they are to be used as probabilities. Naturally, it needs to be checked in any particular case that a solution of the desired type exists.

For any solution   $(\rho^*_{\tau}(j),x^*_j)$ of the system~\eqn{recx}--\eqn{atmost1},  we may set $\rho_{\tau}(j)$ equal to $\rho^*_{\tau}(j)$ for every $0\le j\le \imax$ and   each $\tau\in  \{I,II\} $. Then  $(\sigma(j), q_{\alpha}(j))_{0\le j\le \imax}$ for $\alpha\in\{A,B\}$ can be computed using $x^*_j$ and~\eqn{sigma-x},~\eqn{equaliseclass1} and~\eqn{qA}. Thus, $(\sigma(j), q_{\alpha}(j))_{0\le j\le \imax}$ is a solution to equations~\eqn{equaliseclass} and~\eqn{sigmasGen}, and so by  Lemma~\ref{l:equalised}, the expected number of visits to any given  element  in any \group\ $\state_i$ is $\sigma(i)$.  Recall that the phase finishes as soon as an element of  $\state_0$ is reached. Hence,  every element  in $\state_0$ is reached at most once, and the probability that a given $P\in\state_0$ is reached is equal to $\sigma(0)$,  the same  for all $P\in\state_0$. Thus, we have proved  the following.
   
   \begin{lemma}\lab{lem:uniform} Assume that $(\rho^*_{\tau}(i), x^*_i)$ is a solution to system~\eqn{recx}--\eqn{atmost1}. Set $\rho_{\tau}(i)$ to be $\rho^*_{\tau}(i)$ for every $1\le i\le \imax$ as the type probabilities in the phase. Assume that no rejection occurs during the algorithm. Then  the last  element  visited in the phase  is  distributed uniformly at random in $\state_0$. 
\end{lemma}

%%%%%%%%%%%%%%%%%%%%%%%%%%

\section{Phase 3: double edge reduction}
\lab{sec:double}

 We now turn to giving the explicit construction of the algorithm  \NameA, and its analysis. We will analyse the three phases of the algorithm, essentially applying the general framework in Section~\ref{s:phase} to each phase of \NameA. We  begin with Phase 3, which is the most interesting phase and has been partially described  in Section~\ref{sec:rho}. It was this phase of double-edge reduction that determined the range of $d$ for which {\em DEG} runs efficiently in~\cite{MWgen}. To improve the range of $d$ in~\cite{MWgen} we need to allow more flexible Markov chains, which motivates the new approach in Section~\ref{s:phase}. For phase 3 of \NameA, there will be two types (I and II) of switchings involved, categorized into two classes A and B. The transitions between $(\state_i)_{0\le i\le \imax}$ caused by performing these switchings are exactly as described in Section~\ref{sec:rho}, which lead to system~\eqn{recx}--\eqn{atmost1}.  To complete the specification of phase 3,  we will specify $\Phi_1$, $\state_i$ and $\imax$, and  the set of switchings for this phase. We will analyse these switchings to obtain appropriate parameters $\UB_{\tau}(i)$ and $\LB_{\alpha}(i)$ as required for the system~\eqn{recx}--\eqn{atmost1}. After that,  we will specify values for the variables $ \rho_{I}(i)$ and show that, given these, the whole system has a solution.  Then, in Section~\ref{sec:rejections}, we bound the probability that phase~3, once begun, terminates in a rejection.

\subsection{\ Specifying $\Phi_1$, $\state_i$ and $\imax$}
 
   As described before, if no initial rejection happens, then the algorithm enters phases that sequentially reduces the numbers of loops, triple edges and double edges to zero. Uniformity of the distribution of the output of each phase is guaranteed.  Let $\gamma>0$ be a pre-specified constant.  With the foresight of the definition of $\acceptable$ in Section~\ref{sec:A}, and Lemma~\ref{l:initial}, define 
   \bel{BD}
   \imax=\left\lfloor(1+\gamma)\frac{(d-1)^2}{4}\right\rfloor.
   \ee
 For the analysis here, we may assume that the algorithm enters phase 3 with a pairing $P_0$ uniformly distributed in $\Phi_1$ defined by
   \[
\Phi_1=   \{P\in \Phi:\ \ \mbox{$P$ contains at most $\imax$ double edges, and no other multi-edges or loops} \}.
   \]
       We will verify this assumption in Lemma~\ref{lem:tunif} in Section~\ref{sec:loop-triple} below. Define $\state_i$ to be the set of pairings in $\Phi_1$ containing exactly $i$ double edges. So $\Phi_1=\cup_{0\le i\le \imax}\state_i$.
 
%%%%%%%%%%%%%%%%%%%%%%%%%%%%%%%%%%%%%%%%%%%
\subsection{Definition of the switchings}
\lab{sec:switchings}

As indicated in Section~\ref{sec:rho}, there are two types of switchings, I and II,  in phase~3. 
The two types are different versions of the same basic switching operation, which we call a {\em d-switching},  defined as follows to act upon a pairing $P$ that contains at least one double edge. %\ncv{Some rewording follows in particular here. The old wording kept making my eyes hurt, as I've heard someone say lately. By the way note that it sounds ugly to say a vertex is labelled {\em by} u. That would mean that u does the work of assigning a label. Much better to say {\em with} u, or perhaps omit the preposition if possible. Note also I've avoided calling the edges $x$ and $y$. I'll see how far I can get without it. I note that the first time they are used, it's better to call the edges $u_2v_2$ etc. That's because the count for choices of $x$ would be about $M_1/2$ but calling it $u_1v_1$ imparts the idea that the ends are labelled, hence $M_1$.}  
As in Figure~\ref{f:doubleI}, pick a double edge with  parallel pairs $\{1,2\}$ and $\{3,4\}$, with points 1 and 3   in the same vertex $u_1$ and points 2 and 4 in the same vertex  $v_2$. Also pick another two pairs $\{5,6\}$ and $\{7,8\}$, that represent single  edges  (recalling that we have assumed $\Phi$ contains no pairings with loops), and let $u_2$, $v_2$, $u_3$ and $v_3$ denote the vertices containing the points 5, 6, 7 and 8 respectively.
If all six vertices $u_i$, $v_i$, $1\le i\le 3$, are distinct, then a d-switching replaces the pairs \{1,2\}, \{3,4\}, \{5,6\} and \{7,8\} by \{1,5\}, \{3,7\}, \{2,6\} and \{4,8\}, producing the situation in the right side of Figure~\ref{f:doubleI}.

We emphasise that the specification of the labels in the above and later switching definitions is significant, in that when we come to count the ways that a switching can be applied to a pairing, each distinct valid way of assigning the labels $1, 2, \dots$ and $u_1, v_2,u_2\ldots$ induces   a different switching.

\begin{figure}[htb]

 \hbox{\centerline{\includegraphics[width=9cm]{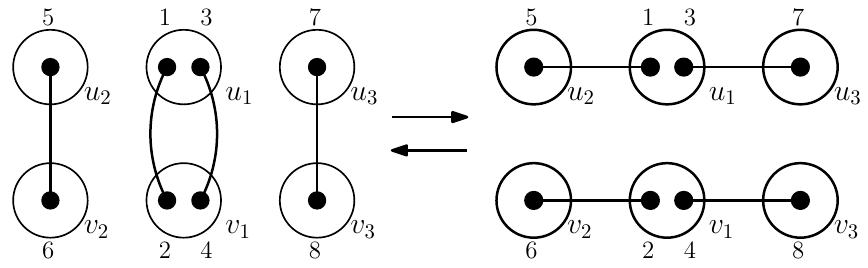}}}
 \caption{\it  A type I switching}

\lab{f:doubleI}

\end{figure}

\ss

\no {\em Type I switching}. If the d-switching operation does not create more than one new double edge, it is a valid type I switching. The {\em class} of a type I switching  depends on how many   new double edges it creates. If none, it is in class A, whilst if there is one, then it is in class B (as in Figure~\ref{f:doubleIb}).  Note that the class of a switching is not chosen; when a random type~I switching is chosen in part (ii) of a switching step, the class is determined after the switching is chosen. For a type I class B switching, there was an existing pair between $u_1$ and $u_i$ or between $v_1$ and $v_i$ ($i=2,3$) before the d-switching was applied. For purposes of counting type I class B switchings, the  existing pair is then labelled $\{9,10\}$, where 9 is in $u_1$ (or $v_1$ as the case may be). Of course, this pair remains after the switching is applied.  See Figure~\ref{f:doubleIb} for an example.\ss

 \begin{figure}[htb]

 \hbox{\centerline{\includegraphics[width=9cm]{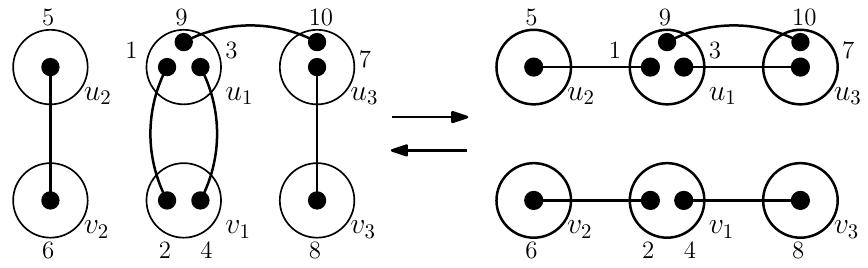}}}
 \caption{\it A type I switching of class B}

\lab{f:doubleIb}

\end{figure}

\no {\em Type II switching}.   If exactly two new double edges, both incident with $u_1$ or both incident with $v_1$, are created by performing a d-switching, then it is a valid type II switching. For  counting purposes, the existing pairs that become part of double edges are labelled $\{9 ,10\}$ and $\{11 ,12\}$, with 9 and 12 both in $u_1$  or both in $v_1$. (See Figure~\ref{f:doubleII1}.) A type II switching is always in class B. Note that the choice of placing point 10  in $u_2$ or in $u_3$, say, leads to different  type II switchings arising from the same d-switching. See Figures~\ref{f:doubleII1} and~\ref{f:doubleII2} for an example. 
   
\begin{figure}[htb]

 \hbox{\centerline{\includegraphics[width=9cm]{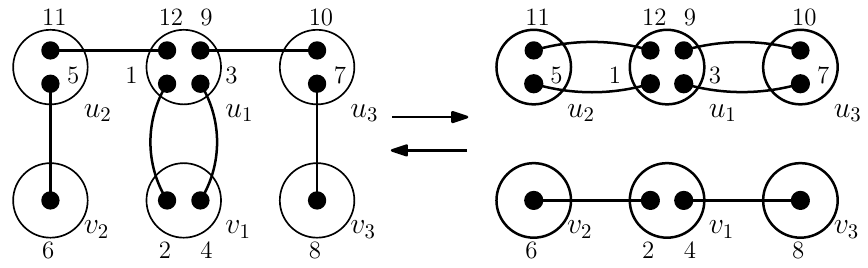}}}
 \caption{\it  A type II switching}

\lab{f:doubleII1}

\end{figure}

\begin{figure}[htb]

 \hbox{\centerline{\includegraphics[width=9cm]{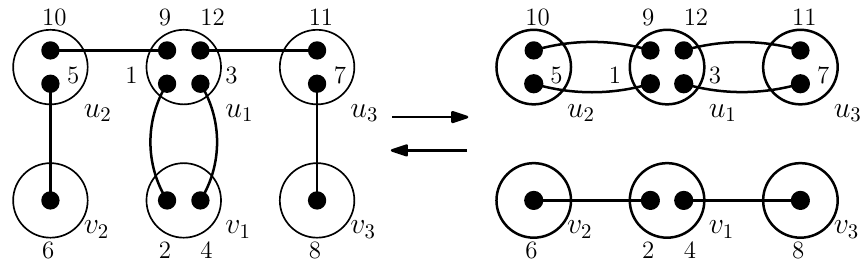}}}
 \caption{\it A different type II switching}

\lab{f:doubleII2}

\end{figure}

We close this subsection by elaborating on the need for the various types and classes of switchings.
The type I class A switching   is the same switching used in~\cite{MWenum} to decrease the number of double edges by one. There, choices of the d-switching are not allowed if $u$ is adjacent to $u_1$ or $u_2$ or if $v$ is adjacent to $v_1$ or $v_2$. This makes it hard to bound the f-rejection probability away from 1 unless $d=O(n^{1/3})$. In order to permit $d$ to be larger, we permit the d-switching if at most one new double edge is created.  These are the type I  class B switchings.  However, permitting these switchings causes a problem.  Each one creates a  2-path containing exactly one double edge. Among all pairings in a \group\ $\state_i$, the number of such 2-paths can vary immensely,  which could lead to a large b-rejection probability.    For instance, consider the extreme case where $d$ is even, $i=(d/2)(d/2+1)$ and $P\in\state_i$ is a pairing such that all the $i$ double edges in $P$ are contained in an isolated clique of order $d/2+1$. Then the number of ways to reach $P$ via a type I, class B switching is 0. If  there were no class B switchings of any other types, this would force $\LB_\alpha(i)$ to equal 0, and hence in part (iii) of the general switching step, all class B switchings would be given b-rejection, rendering them useless.  In order to repair this situation, we  introduce the type II switching (of class B), which boosts the probability of such pairings being reached.

\subsection{Specifying $\UB_{\tau}(i)$ and $\LB_{\alpha}(i)$}
\lab{sec:UBLB}
 For any positive integer $k$, we define $M_k=  n[d]_k$, where $[x]_k$ denotes the  falling factorial, $x(x-1)\cdots (x-k+1)$.  
  Define
\bea
\UB_I(i)&=&4i(M_1-4i)^2  \lab{UBI}\\
\UB_{II}(i)&=&16i(d-1)^2(d-2)(d-3) \lab{UBII}\\
 \LB_A(i)&=& M_2^2- M_2(d-1)(16i + 3d^2+ d+6)   \lab{LBAdouble}\\ 
\LB_B(i)&=& 16i(d-2)M_2-16id(8id+9d^2+3d^3).\lab{LBBdouble}
\eea
\smallskip

The next lemma shows that these  provide the   required bounds on $f_{\tau}(P)$ and $b_{\alpha}(i)$, as given by~\eqn{ParamCond3} and~\eqn{ParamCond4}.
\begin{lemma}\lab{l:mbounds}
In phase 3, for every $P\in \state_i$, 
\begin{enumerate}
\item[(i)] $f_{\tau}(P)\le \UB_{\tau}(i)$ for every $1\le i\le \imax$ and  $\tau\in\{I,II\}$;
\item[(ii)] $b_{\alpha}(i)\ge \LB_{\alpha}(i)$ for every $0\le i\le \imax$ and $\alpha\in\{A,B\}$ such that $(\alpha,i)\neq (A,\imax)$.
\end{enumerate}
\end{lemma}
\begin{proof}
Let $P\in\state_i$. Then $P$ has $i$ double edges. For an upper bound on $f_{\tau}(P)$, there are $4i$ ways to choose a double edge and label its points 1, 2, 3 and 4. The number of choices for each of $u_2v_2$ and $u_3v_3$ is at most $M_1-4i$, since their end points cannot be in any of the  $i$ double edges. Thus $f_I(P)\le\UB_I(i)$.  For the type II switching, we may assume by symmetry that $u_1$, rather than $v_1$, contains points 9 and 12, and that the point 10 lies in vertex $u_3$ rather than $u_2$. These give four choices. There are again $4i$   ways to choose 1, 2, 3 and 4. Then --- recalling that each vertex contains precisely $d$ points ---  there are at most  $(d-2)(d-3)$ ways to choose the two points 9 and 12 in $u_1$ as points 1 and 3 are excluded from the options, at most $d-1$ ways to choose point 5 in $u_2$ and at most $d-1$ ways to choose point 7 in $u_3$. Hence, $f_{II}(P)\le 4\cdot 4i(d-2)(d-3)(d-1)^2=\UB_{II}(i)$.  This shows (i).

For (ii), consider  $P\in \state_i$, where $0\le i\le \imax-1$.  Then $b_A(P)$ is the number of  class A switchings that produce $P$. To determine such a switching, we may choose two 2-paths, \{5,1\}, \{3,7\} and \{6,2\}, \{4,8\}, as in the right hand side of Figure~\ref{f:doubleI}. The number of these arising (with these labels) from the specification of   valid class A switchings is $M_2^2-X$, where $M_2$ is clearly the total number of 2-paths, and $X$ counts the choices such that at least one of the following occurs: 
\begin{enumerate}

  \item[(a)] either of the 2-paths has  a pair contained in  a double edge;
  \item[(b)] both of the 2-paths consist of only pairs representing single edges, and they share at least one vertex;
  \item[(c)] all six vertices $u_i,v_i$ are distinct, and there is an edge between $u_i$ and $v_i$ for some $i\in\{1,2,3\}$.
\end{enumerate}

We only need to bound $X$ from above. For (a), by symmetry we consider the case that $u_1u_3$ is a double edge (see the right side of Figure~\ref{f:doubleIb}) and we multiply the bound of the number of such choices by four. There are $4i$ ways to label point 3 contained in a double edge.  Then, there are $d-1$ ways to choose another point in $u_1$. That bounds the number of 2-paths containing a pair which is part of a double edge. The number of ways to choose the other 2-path is $M_2$. This gives the bound $4\cdot 4i(d-1)M_2=16i(d-1)M_2$ for (a).

For (b), all $u_i$ are distinct and so are all $v_j$. So, there are 9 choices for $i$ and $j$ for the common vertex $u_i=v_j$. Consider the case $u_2=v_2$. There are at most  $M_2$ ways to choose points 5, 1, 3 and 7. There are $d$ ways to choose point 6 (note that point 6 does not need to be distinct from point 5 as this is a bad case counted in $M_2^2$ as well) and then $d-1$ ways to choose point 4. This gives a bound $M_2d(d-1)$. It is easy to see that the same bound holds for each of the other eight cases. So the number of choices for (b) is at most $9M_2d(d-1)$. 

For (c), first consider the case that $u_2v_2$ is an edge. There are at most $M_2$ ways to choose points 5,1,3,7. Since $u_2v_2$ is an edge, there are at most $(d-1)^2$ ways to choose point 6 and then $d-1$ ways to choose point 4. This gives at most $M_2(d-1)^3$ choices for the case that $u_2v_2$ is an edge. By symmetry, the same bound holds for the case that $u_3$ is adjacent to $v_3$. For the case $u_1=v_1$, we have at most $(d-2)(d-1)$ ways to choose point 2 and then at most $d-2$ ways to choose point 4, resulting in a bound $M_2(d-1)(d-2)^2$. Hence, the number of choices for (c) is at most
$M_2(2(d-1)^3+(d-1)(d-2)^2)$.
 Adding the three cases, we get $X\le M_2(d-1)(16i + 3d^2+d+6)$, and so $b_A(P)\ge M_2^2- X\ge \LB_A(i)$.
  
  Finally, to bound $b_B(P)$ from below, we estimate the number of   class B switchings producing $P$. Note that these can be  of either type, I or II. Specifying such a switching involves labelling six vertices and 10 or 12 points, depending on the type of switching, as described in the right hand sides of   Figures~\ref{f:doubleIb} and~\ref{f:doubleII1}. By symmetry, we only discuss the case that point 9 is in $u_1$ and point 10 is in $u_3$, and multiply the resulting bound by 4. Choose a double edge $z$ and choose a point in $z$ and label it $3$ (in $4i$ ways). This determines the points 7, 9 and 10. Choose another point in $u_1$ and label it $1$ (in $d-2$ ways, as points 3 and 9 are already specified in $u_1$). The mate of $1$ is labelled $5$. If $\{1,5\}$ represents   a single edge, this corresponds to a type I switching. Otherwise, it is  in  a  double edge, and this corresponds to a  type II switching (with points 11 and 12 uniquely determined). In each case, another 2-path can be chosen in at most $M_2$ ways. In all, the number of ways to validly choose all these points is $4i(d-2) M_2-Y$ where $Y$ accounts for choices such that at least one of 
 \begin{enumerate}
  \item[(a')] the second 2-path also contains a double edge 
  \item[(b')] the two 2-paths share at least one vertex;
   \end{enumerate}  
  or (c) above holds.
Similar to the previous argument, the number of choices satisfying (a') is at most $2\cdot 4id\cdot  4id$, where factor 2 accounts for whether pair \{6,2\} or \{4,8\} is contained in a double edge. The number satisfying (b') is at most $ 9\cdot 4id\cdot d^2$, and  for (c) it is at most $3\cdot 4id \cdot d^3$.
    Hence, $Y\le 32i^2d^2+36id^3+12 id^4$ and so (recalling the factor 4 from the start)
  \[
  b_B(P)\ge 4\left(4i(d-2)M_2-Y\right)=\LB_B(i).\qed
 \]

\end{proof}

\subsection{Fixing $\rho_{\tau}(i)$}
\lab{sec:rho2}

We assume that $d\ge 2$ in the rest of the paper as the case $d=1$ is trivial: in that case a random pairing immediately gives a random simple 1-regular graph.

 With the definition of the three kinds of switchings (IA, IB and IIB) in Section~\ref{sec:switchings}, the transitions among \groups\ $\state_j$ in the Markov chain is exactly as in the example in Section~\ref{sec:rho}. 
Recalling the description in Section~\ref{sec:rho}, in order to define $\rho_{\tau}(i)$, it is sufficient to specify a solution $(\rho^*_{\tau}(i),x^*_i)$ of the system~\eqn{recx}--\eqn{atmost1}. We have some latitude here because the system is underdetermined. To keep the probability of  t-rejection small in every step, it is desirable to find a solution  such that $1-(\rho^*_I(i)+\rho^*_{II}(i))$ is as small as possible. 
%Let $\eps=cd^2/n^2$ where $c>0$ is a constant. 
We prove that the system  has a unique solution such that $\rho^*_I(i)=1-\eps$ for all $1\le i\le \imax$, for some $0<\eps<1$, at least for $n$ sufficiently large. 
\begin{lemma}\lab{lem:solution} 
Let $0<\gamma <1$ be fixed, which determines $i_1$ via~\eqn{BD}, and   $\eps>0$.  Assume that  each of the following holds: \bea
\frac{4(1+\gamma)(d-1)^2(d-2)(d-3)}{dn\big(dn-7d^2+7d-10 - 4  \gamma(d-1)^2 \big)}&\le& \eps  <1 ;\lab{nLB1}\\
dn-7d^2+7d-10 - 4  \gamma(d-1)^2 &>&0;\lab{nLB2}\\
(d-2)(d-1)n-(2(1+\gamma)d(d-1)^2+9d^2+3d^3)&>&0.\lab{nLB3}
\eea
Then,  with  the values of $\UB_\tau(i)$  and $\LB_\alpha(i)$ defined in Section~\ref{sec:UBLB}, the system~\eqn{recx}--\eqn{atmost1} has a unique solution $(\rho^*_{\tau}(i),x^*_i)$ satisfying $\rho^*_I(i)=1-\eps>0$ for every $1\le i\le \imax$. For fixed $\gamma>0$, and $n$ sufficiently large, there is an $\eps$ satisfying~\eqn{nLB1} with $\eps=O(d^2/n^2)$.
\end{lemma}

\ss
    
\begin{proof}
For $1\le i\le\imax$ we set $\rho^*_I(i)=1-\eps$, which is positive by~\eqn{nLB1}, and necessarily set $\rho^*_I(0)=0$ and $\rho^*_{II}(\imax)=0$ by~\eqn{initial}. Then~\eqn{boundary} determines $x^*_{\imax}$. Then for every $0\le i\le \imax-1$, the value of $x^*_{i}$ is determined recursively, in decreasing order of the value of $i$, by~\eqn{recx} and~\eqn{boundary}.  Equation \eqn{nLB3} ensures that $\LB_B(i)>0$. Moreover, $\LB_B(i)<\UB_{I}(i)$ for every $i$, which we will justify below. It follows then that $x^*_i>0$ for each $i$. From here, each $\rho^*_{II}(i)$ for $0\le i\le \imax-1$ is determined using~\eqn{recrho}, and is moreover strictly positive for  $1\le i\le \imax-1$ because   $x^*_i>0$ for   $0\le i\le\imax$. Note also that $\rho_{II}^*(0)=0$ because $\UB_{II}(0)=0$ by~\eqn{UBII} and this is consistent with~\eqn{initial}.

We have shown that all of $x^*_i$, $\rho^*_{\tau}(i)$ are uniquely determined. It only remains to show that $\LB_B(i)<\UB_{I}(i)$ and $(\rho^*_{\tau}(i),x^*_i)$ satisfies~\eqn{atmost1}. 
 First note that 
\bea 
\LB_A(i)\ge\LB_A(i_1)&\ge& M_2^2-  M_2(d-1)\big(4(1+\gamma)(d-1)^2 + 3d^2+ d+6\big)\nonumber\\
&=& d(d-1)^2n\big(dn-7d^2+7d-10 - 4  \gamma(d-1)^2 \big)\lab{mA}
\eea
which is positive by~\eqn{nLB2}.
 
 By~\eqn{recx},
\be\lab{xratio}
\frac{x^*_{i+1}}{x^*_i}\le \frac{\UB_I(i+1)}{\LB_A(i)(1-\eps)},
\ee
and hence by \eqn{recrho}, 
\[
\rho^*_{II}(i)=(1-\eps)\frac{x^*_{i+1}}{x^*_i}\cdot \frac{\UB_{II}(i)}{\UB_{I}(i+1)}\le  \frac{\UB_{II}(i)}{\LB_A(i)}.
\]
By~\eqn{UBII} and~\eqn{LBAdouble}, and using $i\le \imax\le(1+ \gamma) (d-1)^2/4$,
$$
\rho^*_{II}(i)\le  \frac{4(1+\gamma)(d-1)^4(d-2)(d-3)}{\LB_A(i)},
$$
which is at most $\eps$ by~\eqn{mA} and~\eqn{nLB1}. Hence $\rho^*_{I}(i)+\rho^*_{II}(i)\le 1$ as required for~\eqn{atmost1}. In the first paragraph of this proof it was observed that $\rho^*_{I}(i)$ and $\rho^*_{II}(i)$ are both nonnegative.

To prove $\LB_B(i)<\UB_{I}(i)$, noting that
$\LB_B(i)<16 id^4(n/d-3)$ and $\UB_I(i)>4i(dn-d^2)^2=4id^4(n/d-1)^2$, it is sufficient to prove that
$
4(n/d-3)<(n/d-1)^2.
$
This inequality is trivially true as $(n/d-1)^2-4(n/d-3)=(n/d-3)^2+4>0$.

The last claim in the lemma is immediate since $d=o(\sqrt n)$.
 \qed

\end{proof}
   Note that the proof of Lemma~\ref{lem:solution}   provides a computation scheme for a solution to system~\eqn{recx}--\eqn{recrho} with initial conditions~\eqn{initial}. It is solved by iterated substitutions in  $O(\imax)=O(d^2)$ steps.

 Let $ \gamma>0$ and $\eps$ be chosen to satisfy ~\eqn{nLB1} and~\eqn{nLB2}. Then  $\gamma$ will be used to define $\acceptable$ in Section~\ref{sec:A}. To complete the definition of phase 3 of \NameA, as 
   foreshadowed in Section~\ref{sec:rho}, we then   set $\rho_{\tau}(i)$  for phase 3 to be the solution determined in  Lemma~\ref{lem:solution}. Our main concern in this paper is for large $n$, but there is a technical issue here. If  no such  $ \gamma $ and $\eps$  satisfy~\eqn{nLB1} and~\eqn{nLB2},  \NameA\  would not be technically defined. This is circumvented by defining (in Section~\ref{sec:A})  $i_1=0$ in that case. Then  the system~\eqn{recx}--\eqn{atmost1} has the trivial solution, and phase 3 does, and needs to do, nothing. 
%%%%%%%%%%%%%%%%%%%%%%%%%%%%%%%%%%%%%%%%%%%%%%%%%
\subsection{Probability of  rejection}
\lab{sec:rejections}

Assume that $\gamma$ and $\eps$ are chosen to satisfy~\eqn{nLB1}--\eqn{nLB2}. 
 By Lemma~\ref{lem:solution},~\eqn{recx}--\eqn{atmost1} has a solution and this properly defines $\rho_{\tau}(i)$.

We are assuming that $\gamma$ is fixed,  and $\eps=O(d^2/n^2)$ has been chosen as described at the end of the previous subsection. Our next task is to bound the number of  switching  steps that the algorithm will perform in phase 3, in expectation, per graph generated.   The key points to consider are the distribution of the number of steps in the phase,  and the probability of rejection occurring during the phase.

\begin{lemma}\lab{lem:steps} 
The number of switching steps in phase 3 is $O(\imax)$, both in expectation for sufficiently large $n$, and with probability $1-O( n^{-2})$.
\end{lemma}

\begin{proof}
Let $t\ge \imax$ be an integer.
Assume that phase 3 contains at least $t$ switching steps. Among the first $t$ switching steps, let $x$ denote the number of times that the number of double edges does not decrease. Since it increases by at most 1 per step, and the number of double edges in $P_0$ is at most $\imax$, we have $i_1+x-(t-x)\ge 0$. It follows then that $x\ge (t-i_1)/2$. If a switching does not decrease the number of double edges, it must be either of type I and class B, or of type II. The probability of performing a type II switching for any given switching step in phase 3 (conditional upon the history of the process) is $\rho_{II}(i)$, which is at most $ \eps = O( d^2/n^2)$.     Assume that a type I switching is applied to a pairing $P\in\state_i$ for some $1\le i\le \imax$ in a switching step. It is easy to see that $f_{I}(P)=\Omega(i M_1^2)$   (see a more accurate lower bound for $f_I(P)$ in the proof of Lemma~\ref{lem:drejf} below), and similarly the number of type I class B switchings that can be applied to $P$ is $O(id^2M_1)$.  Thus,   conditional upon the chosen type of switching being I, the probability that the switching is in class B is $O( d^2 /  M_1) =O(d/n)$. Combining the above two cases, the probability that the number of double edges does not decrease in a given switching step is $O(d/n)$. Hence, the probability that $x\ge (t-i_1)/2$ is at most 
\[
\binom{t}{(t-i_1)/2} \big(O(d/n)\big)^{(t-i_1)/2}\le 2^t \big(O(d/n)\big)^{(t-i_1)/2}.
 \] 
 Putting $t=17\imax$, it is clear that the above probability is $O(n^{-2})$ (for all sufficiently large $d$  it is clear that choosing $t=3\imax$ suffices; choosing $t=17\imax$ allows to get the uniform probability bound for cases such as $d=2$). Hence, with probability $1-O(n^{-2})$, phase 3 contains at most $17\imax$ switching steps. The expected number of switching steps in phase 3 is at most  
 \[
2\imax+\sum_{t\ge 2\imax} t 2^t \big(O(d/n)\big)^{t/4} = 2\imax + O(1).\qed
\]

\end{proof}

 \begin{lemma}\lab{lem:drejt} 
 The probability that phase 3 terminates with a t-rejection is $O(d^4/n^2)$.
 \end{lemma}
 \begin{proof}
   By our choice of $\rho_{\tau}(i)$, in each switching step of phase 3, a t-rejection occurs with probability at most $\max_{1\le i\le\imax}(1-\rho_I(i)-\rho_{II}(i))\le 1-\rho_I(i)=\eps= O(d^2/n^2)$. By Lemma~\ref{lem:steps}, the total number of switching steps in phase 3 is $O(\imax)$ with probability $1-O(n^{-2})$. This implies that the probability of a t-rejection during phase 3 is $O(n^{-2}+\imax d^2/n^2)=O(d^4/n^2)$.\qed
   \end{proof}\ss

 Before attempting  to bound the probability of an f-rejection or b-rejection in phase 3, we prove two technical lemmas about properties of a random pairing in $\state_i$.     Given a set of pairs $S$, let $V(S)$ denote the set of vertices that are incident with $S$. We say a set of pairs $K$ {\em excommunicate $S$} if all pairs in $K$ are single  edges, $V(K)\cap V(S)=\emptyset$, and no vertex in $V(K)$ is adjacent to a vertex in $V(S)$.

   %%%%%%%%%
   %%%%%%%%%%%
   \begin{lemma}\lab{lem:1edge} Let $u,v$ be two specified vertices.  Assume $P$ is a random pairing of $\state_i$ conditional on containing $uv$ as a single edge (or a double edge, or that $u$ is not adjacent to $v$).
Let $x$ be a randomly selected pair in $P$. The probability that one end vertex of $x$ is adjacent to $u$,  and the other adjacent to $v$, and all four vertices are distinct, is $O(d^2/n^2)$.
%\item[(b)] The probability that $u$ is incident with another double edge other than $\{u,v\}$ is $O(d^2/n)$.
    \end{lemma}
    %%%%%%%%%%%
    %%%%%%%%%%%
    
   \begin{proof}
The proof is basically the same for all three cases, so it is enough to assume that $uv$ is a double edge.
   Label the pairs in the edge $uv$ as  \{1,2\} and \{3,4\}, with points 1 and 3 in $u$.
   We will use a ``subsidiary'' switching. To select a random pair  $x$, we may let it be the pair containing a randomly selected point.  Let the end points of $x$ be labelled $5$ and $6$, the vertex containing 5   labelled $u_1$ and the one containing 6   labelled $v_1$. We will bound the probability that $u_1$ is adjacent to $u$ and $v_1$ is adjacent to $v$ by $O(d^2/n^2)$. The lemma then follows by symmetry.  So far we have only exposed pairs in the edge $uv$ and pair $x$, and we can assume that all four vertices involved are distinct. All other points are paired uniformly at random in $P$ conditional on $P\in\state_i$.  

Pick a point 7 from $u_1$, a point 8 from $u$, a point 9 from $v_1$ and a point 10 from $v$. We next estimate the probability that $y=\{7,8\}$ and $z=\{9,10\}$ are pairs in $P$, and both represent single edges.

Let $C_1$ be the set of pairings that contain the pairs $\{2i-1,2i\}$ for all $1\le i\le 5$, such that $y$ and $z$ represent single edges.  For any $P\in C_1$, pick sequentially another two single-edge pairs labelled   \{11,12\} and \{13,14\} that each excommunicates  all previously chosen/exposed pairs.  Perform a switching of pairs as in Figure~\ref{f:cycle}.  Then the number of valid switchings is $\Theta(M_1^2)$, since  there are $M_1-O(d^2)= M_1-o(n) \sim M_1 $ options for each of these two single-edge pairs. Let $C$ be the set of pairings in $\C_i$ containing $\{u,v\}$ as a double edge, as well as the pair $x$. Observe that each subsidiary switching produces a pairing in $\C_i$ and indeed in $C$. For any pairing $P'\in C$,  at most one subsidiary switching can produce $P'$, because all points $1\le i\le 10$ are specified and so are points $11\le i\le 14$. Hence, the probability that a random $P\in C$ contains pairs $y$ and $z$ as single-edge pairs is
\[
\frac{|C_1|}{|C|}=O(1/M_1^2).
  \]  
There are at most $d^4$ ways to choose points $7\le i\le 10$. Thus, the probability that $u_1$ is adjacent to $u$ with a single edge and $v_1$ is adjacent to $v$ with a single edge is $O(d^4/M_1^2)=O(d^2/n^2)$.

\begin{figure}[htb]

 \hbox{\centerline{\includegraphics[width=10cm]{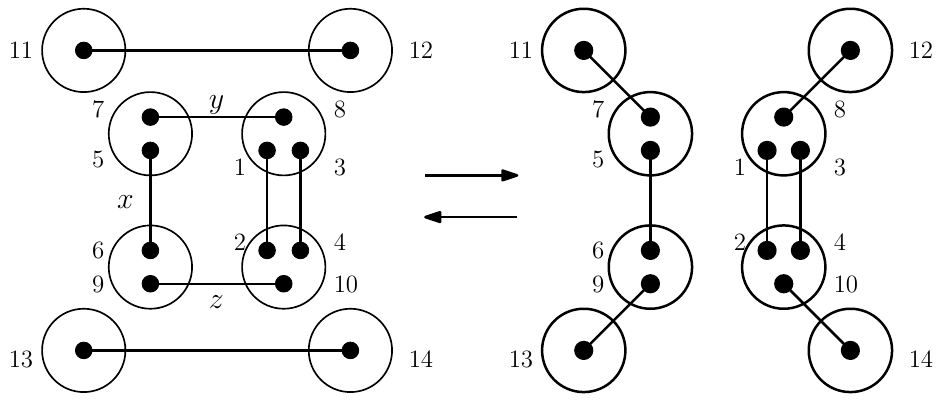}}}
 \caption{\it A subsidiary switching }

\lab{f:cycle}

\end{figure}

Next, we bound the probability that $u_1$ is adjacent to $u$ via a double edge. Pick points 7 and 9 from $u_1$ and 8 and 10 from $u$. We first bound the probability that $\{7,8\}$ and $\{9,10\}$ form a double edge $u_1u$.  Suppose the pairing $P$   does contain $\{7,8\}$ and $\{9,10\}$  (as on the left side of Figure~\ref{f:cycle2}).  Pick sequentially another two 2-paths that each excommunicates all previous chosen/exposed pairs. Switch some of the pairs as in Figure~\ref{f:cycle2}. In this case, the number of possible switchings is   $\Theta(M_2^2)$. On the other hand, the number of switchings that produce any given pairing in $C$ is $O(i)$. This is because all points $1\le i\le 10$ are specified and so are their mates, and there are  $i$ ways to choose the double edge to play the role of the one on the right side of Figure~\ref{f:cycle2} created by the switching. Hence, the probability that \{7,8\} and \{9,10\} are both pairs in $P$ is $O(i/M_2^2)$. There are at most $d^4$ ways to choose points 7, 8, 9, and 10. Hence, the probability that $u_1$ is adjacent to $u$ with a double edge is $O(id^4/M_2^2)=O(i/n^2)$. By symmetry, the probability that $v_1$ is adjacent to $v$ with a double edge is also $O(i/n^2)$.

Noting that $i\le \imax=O(d^2)$, this probability is $O(d^2/n^2)$. This completes the proof.\qed
                   
\begin{figure}[htb]

 \hbox{\centerline{\includegraphics[width=10cm]{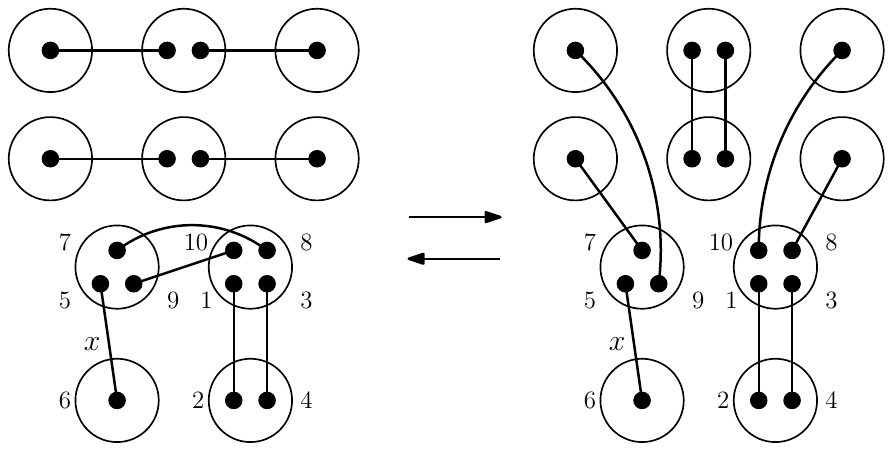}}}
 \caption{\it  Another subsidiary switching}

\lab{f:cycle2}

\end{figure}
                  
   \end{proof}

   %%%%%%%%%%%
   %%%%%%%%%%%
   \begin{lemma}\lab{lem:2double}
   Let $P$ be a random pairing in $\state_i$ with $i\ge 2$. The expected number of pairs of double edges sharing a common vertex is $O(i^2/n)$.
   \end{lemma}
      %%%%%%%%%%%
   %%%%%%%%%%%
   \begin{proof}
   Let $u$, $u_1$ and $u_2$ be three distinct vertices. Pick four points in $u$, two points in $u_1$ and two points in $u_2$. Label these points as in Figure~\ref{f:2double}. We estimate the probability that pairs $\{2i-1,2i\}$ for all $1\le i\le 4$ are in $P$. Let $C_1$ be the set of pairings containing these four pairs. For any $P\in C_1$, consider the following switching. Choose sequentially 
   four 2-paths, each of which excommunicates all previously chosen and exposed pairs. Switch these pairs as in Figure~\ref{f:2double}. There are $\Omega(M_2^4)$ ways to perform such a switching for every pairing in $C_1$, whereas for every pairing $P'\in\state_i$ ($i\ge 2$), the number of switchings producing $P'$ is $O(i^2)$, as all points $1\le i\le 8$ are specified and there are $O(i^2)$ ways to choose the two double edges as on the right side of Figure~\ref{f:2double}. Hence, the probability that $P$ contains pairs $\{2i-1,2i\}$ for all $1\le i\le 4$ is $O(i^2/M_2^4)$. The number of ways to choose $u$, $u_1$ and $u_2$ and then to choose points $1\le i\le 8$ is at most $M_4M_2^2$. By linearity, the expected number of pairs of double edges sharing a common vertex is
   $O(M_4M_2^2\cdot i^2/M_2^4)=O(i^2/n)$.
             
\begin{figure}[htb]

 \hbox{\centerline{\includegraphics[width=12cm]{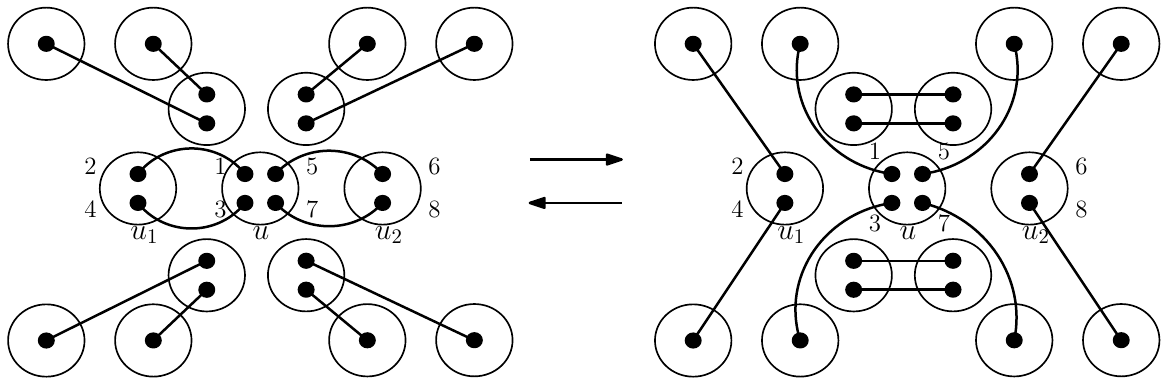}}}
 \caption{\it  Switching away a pair of double edges sharing a common vertex}

\lab{f:2double}

\end{figure}

   \end{proof}
   %%%%%%%%%%%
   %%%%%%%%%%%

 \subsubsection{Probability of an f-rejection}
  Recall that $P_t$ is the pairing obtained after the $t$-th switching step of phase 3. 
For a given pairing $P$, conditional on $P_{t}=P$, we bound the probability an f-rejection occurs at the $(t+1)$-th switching step. Assume $P\in\state_i$; the probability that a type I switching is performed on $P$ and is f-rejected, is
\[
\rho_I(i)\left(1-\frac{f_I(P)}{\UB_I(i)}\right).
\]
On the other hand, the probability that a type II switching is performed on $P$ and is f-rejected is $O(d^2/n^2)$, as $\rho_{II}(i)=O(d^2/n^2)$ by Lemma~\ref{lem:solution}. %Therefore, the probability that an f-rejection happens at the $(t+1)$-th switching step, conditional on $P_t=P$, is 
%\be\lab{fprob1}
%\rho_I(i)\left(1-\frac{f_I(P)}{\UB_I(i)}\right)+O(d^2/n^2).
%\ee
Taking the union bound over all switching steps, the probability of an f-rejection during phase 3 is at most
\[
\sum_{t\ge 0} \sum_{1\le i\le \imax} \sum_{P\in \state_i} \pr(P_t=P) \left(\rho_I(i)\left(1-\frac{f_I(P)}{\UB_I(i)}\right)+O(d^2/n^2)\right). 
\]
Note that $\sum_{t\ge 0}\pr(P_t=P)$ is the expected number of times that $P$ is reached, which is $\sigma(i)$ for every $P\in\state_i$ by Lemma~\ref{l:equalised}. Hence, the above probability is
\be\lab{fprob}
\sum_{1\le i\le \imax}\sum_{P\in\state_i}  \sigma(i) \left(\rho_I(i)\left(1-\frac{f_I(P)}{\UB_I(i)}\right)\right)+\sum_{1\le i\le \imax} |\state_i| \sigma(i)\cdot O(d^2/n^2). 
\ee
Since $\rho_I(i)\le 1$ always, and since by Lemma~\ref{lem:steps},
 \be\lab{totalsteps}
 \sum_{1\le i\le \imax} |\state_i| \sigma(i) =O(\imax)=O(d^2),
   \ee  
 the probability of an f-rejection during phase 3 is at most
 \be\lab{fprob2}
 \sum_{1\le i\le \imax} \sum_{P\in\state_i}\sigma(i) \left(1-\frac{f_I(P)}{\UB_I(i)}\right)+O(d^4/n^2).
   \ee  
We will prove the following lemma by bounding the first term of~\eqn{fprob2}.

      \begin{lemma} \lab{lem:drejf}
 The probability of an f-rejection during phase 3 is $O(d^2/n)$.
  \end{lemma}
  %%%%%%%%
  %%%%%%%%%
 \begin{proof}
 By~\eqn{fprob2}, it is sufficient to bound $\sum_{1\le i\le \imax} \sum_{P\in\state_i}\sigma(i) \left(1-f_I(P)/\UB_I(i)\right)$ by $O(d^2/n)$. Note that
   \bean
\sum_{1\le i\le \imax} \sum_{P\in\state_i}\sigma(i) \left(1-\frac{f_I(P)}{\UB_I(i)}\right)=\sum_{1\le i\le \imax} \sigma(i)|\state_i|\cdot \ex\left(1-\frac{f_I(P)}{\UB_I(i)}\right),
 \eean
where the above expectation is taken on a random pairing $P\in\state_i$. It will be sufficient to show that this expectation is $O(1/n)$ for every $1\le i\le\imax$, since then
\[
\sum_{1\le i\le \imax} \sigma(i)|\state_i|\cdot \ex\left(1-\frac{f_I(P)}{\UB_I(i)}\right)=O(\imax /n)=O(d^2/n)
\]
and the lemma follows.
 
Assume $P\in\state_i$. Let $X=\UB_{I}(i)-f_I(P)$. 
Pick and label pairs and vertices as in the description of the type I switching (see Figure~\ref{f:doubleI}). Here, $X$ counts choices and labellings such that any of the following occur:
 \begin{enumerate}
  %\item[(a)] one of $u_2v_2$ and $u_3v_3$ is a double edge;
 \item[(a)] the six vertices are not all distinct;
 \item[(b)] the vertices are distinct and at least two new double edges are created after performing the switching;
 \item[(c)] some triple edge is created by the switching because there is a double edge already  between $u_1$ and $u_2$ or $u_3$, or between $v_1$ and $v_2$ or $v_3$.
 \end{enumerate}
It is easy to see that the number of choices for (a) is $O(idM_1)$. 
 For (b), there are two sub-cases (we use $u\sim v$ to denote that $u$ is adjacent to $v$).
 \begin{enumerate}
 \item[(b1)] $u_1\sim u_2$ and $u_1\sim u_3$, or $v_1\sim v_2$ and $v_1\sim v_3$, or $u_1\sim u_2$ and $v_1\sim v_3$, or $u_1\sim u_3$ and $v_1\sim v_2$;
  \item[(b2)]  $u_1\sim u_3$ and $v_1\sim v_3$, or $u_1\sim u_2$ and $v_1\sim v_2$.
 \end{enumerate}
It is straightforward to see that the number of choices for (b1) is at most $O(id^4)$.  For (b2), we consider a random $P\in\state_i$ which contains $u_1v_1$ as a double edge.   Let $x$ be a random pair chosen uniformly from $P$; with $u_2$ and $v_2$ denoting its end vertices. By Lemma~\ref{lem:1edge}, the probability that $u_1$ is adjacent to $u_2$ and $v_1$ is adjacent to $v_2$ is $O(d^2/n^2)$.  
  Hence, the expected number of choices for (b2) is $O(iM_1^2\cdot d^2/n^2)$, as there are $O(i)$ ways to choose $u_1v_1$ and label their end points, $M_1^2$ ways to sequentially choose two random pairs, uniformly at random and with repetition, and by the above argument, the probability that one of these two pairs forcing (b2) is $O(d^2/n^2)$. By~\eqn{UBI} and since $i\le \imax=O(d^2)=o(M_1)$, $iM_1^2\cdot d^2/n^2=\UB_I(i)\cdot O(d^2/n^2)$.
 
 We can count the possibilities for (c) by first choosing a vertex (for $u_1$ or $v_1$) incident with two double edges --- by Lemma~\ref{lem:2double} there are $O(i^2/n)$ ways to do this in expectation --- and then choosing the remaining pairs to be switched, in $O(dM_1)$ ways. For a random $P\in\state_i$, the expected number of type I switchings satisfying (c) is therefore $O(i^2dM_1/n) =  o(i dM_1 )$ since $i=O(d^2)=o(n)$.

Combining the three bounds, we find that $\ex X=O(idM_1+ id^4)+ \UB_I(i)\cdot O(d^2/n^2)$ and so since $d^2=o(n)$ we have
 $\ex\left(1-f_I(P)/\UB_I(i)\right)$ is  
 $ O(idM_1+ id^4)/\UB_I(i)+O\left(d^2/n^2\right)=O(1/n)$,
 as desired.\qed
 \end{proof}
 
\subsubsection{Probability of a b-rejection}
Let $\state^-_{\tau,\alpha}(P)$ denote the set of pairings that can produce $P$ using a type $\tau$, class $\alpha$ switching.\lab{def:state-}

Now, we bound the probability of a b-rejection during phase 3. 
Given $P\in\state_i$, we consider a pairing $P'\in \state^-_{I,A}(P)$. Then $P'$ must be in $\state_{i+1}$. Conditional on $P_{t-1}=P'$, the probability that the switching $S$, which converts $P'$ to $P$, is chosen and is not f-rejected at the $t$-th switching step is
\[
\frac{\rho_I(i+1)}{\UB_I(i+1)}.
\]
Conditional on that, the probability that $P$ is b-rejected is
$1-\LB_A(i)/b_A(P)$. Similarly, for any $P'\in \state^-_{I,B}(P)\cup \state^-_{II,B}(P)$, conditional on $P_{t-1}=P'$, the probability that the switching converting $P'$ to $P$ is chosen at the $t$-th switching and is not f-rejected but is b-rejected is
\[
\frac{\rho_I(i)}{\UB_I(i)}\left(1-\frac{\LB_B(i)}{b_B(P)}\right),
\]
if $P'\in \state^-_{I,B}(P)$, and is
\[
\frac{\rho_{II}(i-1)}{\UB_{II}(i-1)}\left(1-\frac{\LB_B(i)}{b_B(P)}\right).
\]
if
$P'\in \state^-_{II,B}(P)$.
Hence, by the union bound, the probability that a b-rejection occurs during phase 3 is at most
\bea
&&\sum_{t\ge 1} \sum_{0\le i\le \imax-1} \sum_{P\in \state_i}\sum_{P'\in\state^-_{I,A}(P)} \pr(P_{t-1}=P') \frac{\rho_I(i+1)}{\UB_I(i+1)}\left(1-\frac{\LB_A(i)}{b_A(P)}\right)\non\\
&&\hspace{0.3cm}+\sum_{t\ge 1} \sum_{1\le i\le \imax} \sum_{P\in \state_i}\left(1-\frac{\LB_B(i)}{b_B(P)}\right)\left(\sum_{P'\in\state^-_{I,B}(P)} \pr(P_{t-1}=P') \frac{\rho_I(i)}{\UB_I(i)}\right.\non\\
&&\qquad\qquad\qquad\qquad\qquad\qquad\qquad\qquad\left.+\sum_{P'\in\state^-_{II,B}(P)} \pr(P_{t-1}=P') \frac{\rho_{II}(i-1)}{\UB_{II}(i-1)}\right)\non\\
&&=\sum_{0\le i\le \imax-1} \sum_{P\in \state_i}\sum_{P'\in\state^-_{I,A}(P)} \sigma(i+1) \frac{\rho_I(i+1)}{\UB_I(i+1)}\left(1-\frac{\LB_A(i)}{b_A(P)}\right)\non\\
&&\hspace{0.3cm}+\sum_{1\le i\le \imax} \sum_{P\in \state_i}\left(1-\frac{\LB_B(i)}{b_B(P)}\right)\left(\sum_{P'\in\state^-_{I,B}(P)} \sigma(i) \frac{\rho_I(i)}{\UB_I(i)}+\sum_{P'\in\state^-_{II,B}(P)} \sigma(i-1) \frac{\rho_{II}(i-1)}{\UB_{II}(i-1)}\right).\non 
 \eea
 
 By~\eqn{equaliseclass1}, $\sigma(i) \rho_I(i)/\UB_I(i)=\sigma(i-1) \rho_{II}(i-1)/\UB_{II}(i-1)$ for every $1\le i\le \imax$. So, using $\rho_I(i)\le 1$,
 the probability of a b-rejection during phase 3 is at most
 \bea
&&\sum_{0\le i\le \imax-1} \sum_{P\in \state_i}\left(\sum_{P'\in\state^-_{I,A}(P)}  \frac{\sigma(i+1)}{\UB_I(i+1)}\left(1-\frac{\LB_A(i)}{b_A(P)}\right)\right)\non\\
&&+\sum_{1\le i\le \imax} \sum_{P\in \state_i}\left(\sum_{P'\in\state^-_{I,B}(P)\cup \state^-_{II,B}(P)}  \frac{\sigma(i)}{\UB_I(i)}\left(1-\frac{\LB_B(i)}{b_B(P)}\right)\right).\lab{bprob}
 \eea
By the definition of $b_{\alpha}(P)$ for $\alpha\in\{A,B\}$, 
\[
|\state^-_{I,A}(P)|=b_A(P),\quad |\state^-_{I,B}(P)\cup \state^-_{II,B}(P)|=b_B(P).
\]
Then, the above expression equals
\bel{bprob3}
\sum_{0\le i\le \imax-1} \sum_{P\in \state_i}\frac{\sigma(i+1)}{\UB_I(i+1)}\Big(b_A(P)-\LB_A(i)\Big)+\sum_{1\le i\le \imax} \sum_{P\in \state_i}\frac{\sigma(i)}{\UB_I(i)}\Big(b_B(P)-\LB_B(i)\Big).
\ee
%Then the probability of a b-rejection is at most
%\[
%O\left(\sum_{1\le i\le \imax} \sum_{P\in\state_i} \left(\frac{\sigma(i+1)M_2^2}{\UB_I(i+1)}\left(1-\frac{\LB_A(i)}{b_A(P)}\right)+ \frac{\sigma(i) idM_2}{\UB_I(i)}\left(1-\frac{\LB_B(i)}{b_B(P)}\right)\right)\right).
%\]
By~\eqn{xratio} and recalling that $x_j^*=\sigma(j)/|\Phi_1|$, 
\bel{bound1}
\frac{\sigma(i+1)}{\UB_I(i+1)} =O\left(\frac{\sigma(i)}{\LB_A(i)}\right),
\ee                              
whereas by~\eqn{UBI} and~\eqn{LBBdouble},
\bel{bound2}
\frac{\sigma(i)}{\UB_I(i)}=\frac{\sigma(i)}{\LB_B(i)}\frac{\LB_B(i)}{\UB_I(i)}=\frac{\sigma(i)}{\LB_B(i)} \cdot O\left(\frac{idM_2}{iM_1^2}\right)=\frac{\sigma(i)}{\LB_B(i)}\cdot O(d/n).
\ee
Substituting bounds~\eqn{bound1} and~\eqn{bound2} into~\eqn{bprob3} yields
 \bea
&&\sum_{0\le i\le \imax} \sum_{P\in\state_i} O(\sigma(i)) \left(\frac{b_A(P)-\LB_A(i)}{\LB_A(i)}+ (d/n) \left(\frac{b_B(P)-\LB_B(i)}{\LB_B(i)}\right)\right)\non\\
&&\hspace{0.5cm} =\sum_{1\le i\le \imax}  O(\sigma(i)|\state_i|) \left(\frac{\ex b_A(P)-\LB_A(i)}{\LB_A(i)}+ (d/n) \left(\frac{\ex b_B(P)-\LB_B(i)}{\LB_B(i)}\right)\right),\lab{bprob2}
\eea          
by noting that $b_{A}(P)=0$ for $P\in\state_{\imax}$ and $b_B(P)=\LB_B(0)=0$ for $P\in\state_0$, and by resetting $\LB_{A}(\imax)=0$ as the first term in~\eqn{bprob3} does not include $i=\imax$. By bounding these expectations we will prove the following lemma.
   
 \begin{lemma} \lab{lem:drejb}
 The probability of a b-rejection during phase 3 is $O(d^2/n)$.
  \end{lemma}
 \begin{proof} %We estimate the probability that a b-rejection occurs in a given step.
%Assume $P_t\in\state_i$ is obtained after a switching; $P$ can be obtained by 
%\begin{enumerate}
%\item[(i)] a type I class A switching from a paring in $\state_{i+1}$, or 
%\item[(ii)] a type I class B switching from a pairing in $\state_i$, or 
%\item[(iii)] a type II switching from a pairing in $\state_{i-1}$, if $i\ge 2$.   
%\end{enumerate}
%A class B switching is perform in this step either if a type II switching or a type I class B switching is performed. As $\rho_{II}(i)=O(d^2/n^2)$ for every $i$, the probability that $P$ is obtained via a type II switching is $O(d^2/n^2)$. It is easy to see that when a random type I switching is performed, the probability that it is in class B is $O(d^2/M_1)=O(d/n)$. 
%Hence,  in any given step, the probability that a class B switching is performed is $O(d/n)$. Assume that $P$ is obtained via a class B switching.    
%By~\eqn{bounds},
%\[
% b_A(P)\ge \LB_A(i),\quad b_B(P)\ge \LB_B(i).
%\]
%It is easy to see that $b_B(P)\le 16i(d-2)M_2$. So, $P$ is b-rejected with probability $1-\LB_B(i)/b_B(P)=O(id^4/idM_2)=O(d/n)$. Since the probability that a class B switching is performed in a step is $O(d/n)$,  the probability that a b-rejection happens in a given step because of a class B switching is $O(d^2/n^2)$.
We first bound $\ex b_A(P)$ from above for a random $P\in\state_i$. Clearly $M_2^2$ is an upper bound. However, we can exclude from it the choices of the pair of 2-paths such that
\begin{enumerate}
\item[(a)] Exactly one of these four pairs is contained in a double edge and the two paths do not share any vertex;
\item[(b)] All of the four pairs in the 2-paths represent single edges, and the two paths are vertex-disjoint, and $u_i$ is adjacent to $v_i$ for exactly one of $i=1,2,3$.
\end{enumerate}
Clearly, the sets of pairings counted in (a) and (b) do not intersect. Let $X_a$ and $X_b$ denote the sizes of these two sets. Then $b_A(P)\le M_2^2-X_a-X_b$. We wish to estimate lower bounds for $X_a$ and $X_b$.

For (a), consider the case that the pair \{3,7\} is contained in a double edge, as on the right side of Figure~\ref{f:doubleIb}. There are $4i$ ways to choose point 3. That choice automatically sets points 7,9 and 10. Then, there are $(d-2)$ ways to pick point 1, which sets point 5 as well.  Given the choice of the first 2-path, there are $M_2-O(d^2+id)$ ways to choose the other 2-path so that it does not use any vertex $u_i$, and it does not use a pair contained in a double edge. Hence, by symmetry,
$X_a=16i(d-2)(M_2-O(d^2+id))-O(X')$, where $X'$ counts choices where one of the two 2-paths uses two pairs such that each pair is contained in a double edge. By Lemma~\ref{lem:2double}, $\ex(X')=O(M_2 i^2/n)=O(d^6)$, as $i\le\imax=O(d^2)$, and the expected number of choices for the 2-paths with both pairs contained in double edges is $O(i^2/n)$, whereas $M_2$ bounds the number of choices for the other 2-path. Now, we have
$\ex X_a\ge 16i(d-2)(M_2-O(d^2+id))-O(d^6)$.

For (b), there are $M_2$ ways to choose the first path and then approximately $2(d-1)^3+(d-2)^2(d-1)$ ways to choose the other 2-path. So, $X_b=M_2(2(d-1)^3+(d-2)^2(d-1))-Y$, where $Y$ counts the choices of the two 2-paths such that 
 \begin{enumerate}
 \item[(c)] the two paths are not vertex disjoint, or
 \item[(d)] they are vertex-disjoint, and one of the two paths contains a double edge and $u_i\sim v_i$ for some $i$, or
 \item[(e)] they are vertex-disjoint, and at least two of $u_iv_i$ are edges. 
 \end{enumerate}
 The number of choices for (c) is $O(M_2d^2)$; the number of choices for (d) is $O(id^4)$. For (e),  first consider the case that $u_1v_1$ and $u_2v_2$ are edges. There are $O(M_1)$ choices for the edge $u_1v_1$. Consider a random pairing $P\in\state_i$ that contains $u_1v_1$ as an edge. Pick a random pair $x$ in $P$. Denote its end vertices by $u_2$ and $v_2$. By Lemma~\ref{lem:1edge}, the probability that $u_1\sim u_2$ and $v_1\sim v_2$ is $O(d^2/n^2)$. Thus, the expected number of such choices for $u_1,u_2,v_1,v_2$ and pairs between them is $O(M_1^2\cdot d^2/n^2)$. After fixing these, there are $O(d^2)$ ways to choose edges $u_1u_3$ and $v_1v_3$. This gives $O(M_1^2d^4/n^2)=O(d^6)$ for the number of choices such that $u_1\sim v_1$ and $u_i\sim v_i$ for some $i=2,3$ by symmetry. Lastly, consider the case that $u_i\sim v_i$ for $i=2$ and 3. There are at most $n^2$ ways to choose $u_1$ and $v_1$. Then, randomly pick 2 pairs and denote their end vertices by $u_2$ and $v_2$, and $u_3$ and $v_3$ respectively. A trivial adaptation of the proof as Lemma~\ref{lem:1edge} shows that the probability that the four edges $u_1u_i$ and $v_1v_i$ for $i=2,3$ are present is $O(d^4/n^4)$. Hence, the number of choices in this case is
 $O(n^2\cdot M_1^2\cdot d^4/n^4)=O(d^6)$.
 This gives
 \[
 \ex X_b=2M_2(d-1)^3+M_2(d-2)^2(d-1)-O(M_2d^2+id^4+d^6)=3M_2d^3+O(M_2d^2+d^6),
 \]
 using that $i\le \imax=O(d^2)$.

 Combining (a) and (b), 
$\ex b_A(P)\le M_2^2-16idM_2-3M_2d^3+O(d^6+M_2d^2)$.
Recall from~\eqn{LBAdouble} that $\LB_A(i)= M_2^2-M_2d(16i+9d+3d^2)$. So,
$\ex b_A(P)-\LB_A(i)=O(d^6+M_2d^2)$.
%If $P$ is obtained via a class A switching, it is b-rejected with probability  
%\bean
%1-\LB_A(i)/b_A(P)&\le &\frac{b_A(P)-\LB_A(i)}{\LB_A(P)}\\
%&=&O\left(\frac{d^6+X_2+M_2d^2}{M_2^2}\right)=O\left(\frac{1}%{n}+\frac{X_2}{M_2^2}\right).
%\eean
% In order to bound $\ex(b_A(P)-\LB_A(i))$, it only remains to obtain an upper bound for $\ex(X_2)$. A trivial bound is $O(idM_2)$. However,  $\ex X_2$ is much smaller.
% Consider a random $P\in \state_i$. 
% By Lemma~\ref{lem:2double}, $\ex(X_2)=O(i^2M_2/n)=O(d^4M_2/n)$, where $M_2$ bounds the number of ways to choose the other 2-path. Hence, 
%\[
%\ex(b_A(P)-\LB_A(i))=O(d^6+d^4M_2/n+M_2d^2).
%\]

It is easy to see that $b_B(P)\le 16i(d-2)M_2$ for every $P\in\state_i$. Hence, by~\eqn{LBBdouble},
\[
\ex (b_B(P)-\LB_{B}(i)) =O(id^4).
\]
By~\eqn{bprob2}, the probability of a b-rejection during phase 3 is
\bean
&&\sum_{1\le i\le \imax}  O(\sigma(i)|\state_i|) \left(\frac{d^6+M_2d^2}{\LB_A(i)}+ (d/n) \left(\frac{id^4}{\LB_B(i)}\right)\right).
\eean 
It is easy to see that $\LB_A(i)=\Omega(M_2^2)$ and $\LB_B(i)=\Omega(idM_2)$ by~\eqn{LBBdouble} and the assumption that $d^2=O(n)$. So the above probability is bounded by
\[
O(1/n)\sum_{1\le i\le \imax} \sigma(i)|\state_i|=O(\imax/n)=O(d^2/n),
\]
by~\eqn{totalsteps}. The lemma follows.\qed
 \end{proof}

%%%%%%%%%
%%%%%%%%%%

\section{Initial rejection: defining $\acceptable$}
\lab{sec:A}

Here we define $\acceptable$, after some preliminaries. 
It is easy to see that
\[
|\Phi|=(M_1-1)!!:=\prod_{i=0}^{M_1/2-1}(M_1- 2i-1), 
\]
and the probability that a random $P\in\Phi$ contains a given set of $s$ pairs  is
\be\lab{pairprob}
\frac{(M_1-2s-1)!!}{(M_1-1)!!}=\prod_{i=0}^{s-1}\frac{1}{M_1-1-2i}.
 \ee 
Let $D(P)$, $L(P)$ and $T(P)$ denote the number of double edges, loops and triple edges, respectively, in a pairing in $P\in \Phi$. Define 
\[
 \acceptable = \{P\in \Phi: \ D(P)\le B_D,\, L(P)\le B_L,\, T(P)\le B_T\} 
\]
 where
 \be
B_L=\max\{\lceil \log dn\rceil,4d\},\quad B_T=\max\{\lceil \log dn\rceil, \lceil d^3/n\rceil\},\quad \ B_D=\left\lfloor(1+\gamma)\frac{(d-1)^2}{4}\right\rfloor, \lab{initialbounds}
\ee
noting that the definition of $B_D$ agrees with the definition of $\imax$ in~\eqn{BD}, and subject to the following proviso. If $\gamma$ is such that no $\eps$ satisfies the conditions in Lemma~\ref{lem:solution} for phase 3, we replace the bound on $D(P)$ by the condition $D(P)=0$. That is, $i_1=0$ for phase 3. (This will increase the probability of initial rejection, but only for $n=O(d)$, as discussed after Lemma~\ref{lem:solution}.)

%\jcomk{This following lemma has perhaps appeared in~\cite{MWenum}. Should add a reference here. We may just include the proof for completeness.}
  \begin{lemma}\lab{l:initial} Assume  $d=o(\sqrt{n})$. For any $\gamma>0$, the probability of an initial rejection is at most $1/(1+\gamma)+o(1)$.
  \end{lemma}
\proof By~\cite[Lemma 3.2]{MWenum}, the probability that a random pairing $P\in \Phi$ contains more than $B_L$ single loops or more than $B_T$ triple edges, or any multi-edge other than double or triple edges, is $O(d^2/n)=o(1)$. We also have 
\[
\ex D=\sum_{1\le i<j\le n}\frac{\binom{d}{2}\binom{d}{2}2!}{(M_1-1)(M_1-3)}= (1+o(1))(d-1)^2/4.
\]
By Markov's inequality, the probability that a random pairing $\P\in \Phi$ has at least $(1+\gamma)(d-1)^2/4$ double edges is at most $1/(1+\gamma)+o(1)$. The lemma follows, recalling again  that by the remarks after Lemma~\ref{lem:solution}, the condition $D(P)=0$ will only be imposed for $n=O(d)$.\qed
% \smallskip

%The following corollary is immediate.
 
%\begin{cor}\lab{cor:initial}
%The probability of an initial rejection is at most $1/2+o(1)$.
%\end{cor}

\section{Reduction of loops and triple edges}
\lab{sec:loop-triple}

Let $P_0$ be a random pairing in $\Phi$. If $P_0\in \acceptable$, then \NameA\ sequentially enters phases 1 and 2 where the loops and triple edges will be switched away, one at a time. In each of these two phases, all switchings have the same type and   class. Moreover, these are the same switchings as used for enumeration purposes in~\cite{MWenum}. In view of the remark below Lemma~\ref{l:equalised}, phase 1 will be the same as the corresponding phase in the algorithm {\em DEG} of~\cite{MWgen}. Phase 2 is similar to phase 1, just with different switchings. (For the range of $d$ considered in~\cite{MWgen} there are no triple edges in pairings in $\acceptableMW$ and so {\em DEG} has no phase of reduction of triple edges.)  
 
We define the switchings  and analyse the rejection probabilities  for phases 1 and 2 in Sections~\ref{loopreduction} and~\ref{triplereduction} respectively. This analysis is similar to that in~\cite{MWgen}.   

\subsection{Phase 1: loop reduction}
\lab{loopreduction}
In this phase, only one type of switching is used, which is defined as follows. 

{\em Type L switching.} Let $P$ be a pairing containing at least one loop. Pick a loop from $P$ and label its end points by 1 and 2. Pick another two pairs $y$ and $z$, both of which represent single non-loop edges, and label their end points by 3 and 4, and 5 and 6 respectively. Let the vertices containing these points be labelled as in Figure~\ref{f:loop}. The type L switching replaces these pairs by \{1,3\}, \{4,6\} and \{2,5\}.  This switching is valid only if the five vertices $v_i$ ($1\le i\le 5$) are all distinct, and there are no new multi-edges or loops created after its performance.  All type L switchings are in class $A$.

 \begin{figure}[htb]

 \hbox{\centerline{\includegraphics[width=9cm]{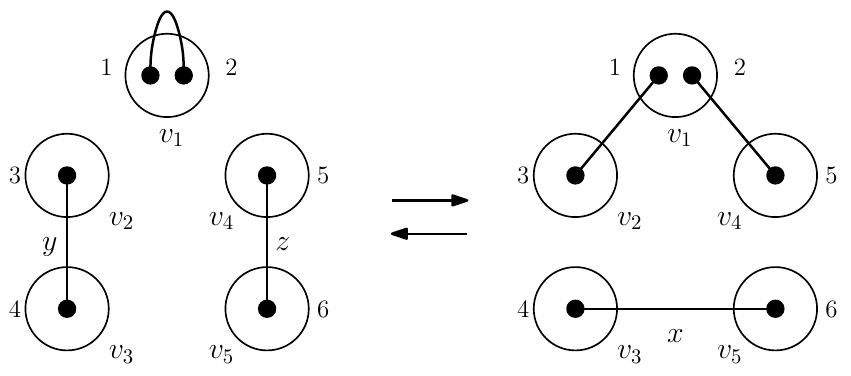}}}
 \caption{\it   type L switching}

\lab{f:loop}

\end{figure}

 Note that a type L switching always reduces the number of loops in a pairing by one. In this phase, $\state_i$ is the set of pairings in $\acceptable$ with exactly $i$ loops, and  $\imax$ is set equal to $B_L$ defined in~\eqn{initialbounds}. Here, $\Phi_1=\acceptable$, which is $\cup_{0\le i\le \imax} \state_i$. 
We set $\rho_{L}(i)=1$ for every $1\le i\le \imax$.

Define
\bea
\UB_{L}(i)&=&2iM_1^2\non\\
\LB_{A}(i)&=&M_2M_1- 2dM_1(3i+6B_D+9B_T+3d+d^2).\lab{loopbound}
\eea

\begin{lemma}
During the   loop reduction phase, for each $P\in \state_i$, $f_{L}(P)\le \UB_{L}(i)$ and $b_{A}(P)\ge \LB_{A}(i)$.
\end{lemma}
\begin{proof}
We first prove the upper bound. Let $P\in\state_i$. The number of ways to choose a loop and label its end points in $P$ is $2i$. Then, the number of ways to choose another two pairs and label their end points is at most $M_1^2$. This shows that $f_{L}(P)\le \UB_L(i)$ for any $P\in\state_i$. 

Next we prove the lower bound. Given $P\in\state_i$, $b_A(P)$ is the number of type L switchings that produce $P$. To count such switchings, we need to specify a 2-path and label the points inside by \{3,1\} and \{2,5\}, together with another pair representing a single edge, whose end points are labelled by $4$ and $6$, and label all vertices containing these points as on the right side of Figure~\ref{f:loop}. Moreover, the five vertices must be distinct. There are $M_2$ ways to pick the 2-path and $M_1$ ways to pick the other pair. We need to exclude every choice in which at least one of the following holds:
\begin{enumerate}
\item[(a)] the 2-path contains a loop, or a double edge, or a triple edge;
\item[(b)] pair $\{4,6\}$ is a loop or is contained in a multi-edge;
\item[(c)] pair $\{4,6\}$ share a vertex with the 2-path;
\item[(d)] $v_2$ and $v_3$ are adjacent or $v_4$ and $v_5$ are adjacent. 
\end{enumerate}

We first show that the number of choices for (a) is at most $2(2id M_1+ 4B_DdM_1+ 6B_Td M_1)$,  starting with the choices in which  \{1,3\} is a loop. There are $2i$ ways to choose a loop and label their end points by 1 and 3. Then, there are at most $d$ ways to choose $2$ and $5$ as $v_1$, the vertex containing point $1$ has degree $d$. So the number of ways to choose the 2-path and $x$ so that \{1,3\} is a loop is at most $2idM_1$. Similarly, the number of ways to choose such a path and a pair so that $\{1,3\}$ is contained in a double edge or a triple edge is at most $4B_D dM_1$ and $6B_Td M_1$ respectively as there are at most $B_D$ double edges and $B_T$ triple edges. The extra factor $2$ accounts for the case that \{2,5\} is a loop or is contained in a double or triple edge.
Hence, the number of choices for (a) is at most $2dM_1(2i+4B_D+6B_T)$.

The number of choices for (b) is at most $M_2(2i+4B_D+6B_T)$, for (c)  is at most $6M_2 d$ and for (d) is at most $2M_2d^2$. Hence, using $M_2\le dM_1$,
\bean
b_A(P)&\ge& M_2M_1-2dM_1(2i+4B_D+6B_T)-M_2(2i+4B_D+6B_T+6d+2d^2)\\
&\ge&M_2M_1- 2dM_1(3i+6B_D+9B_T+3d+d^2)=\LB_A(i).\qed
\eean

\end{proof}

Let $\Phi'=\{P\in \acceptable:\ L(P)=0\}$.
\begin{lemma}\lab{lem:lunif}
Assume that the algorithm is not rejected initially or during the   loop reduction phase  and that $P'$ is the output after phase 1.  Then  $P'$ is uniformly distributed in $\Phi'$. 
\end{lemma}
\begin{proof}
Let $P_t$ be the pairing obtained after the $t$-th switching step in phase 1. We prove by induction that $P_t$ is uniformly distributed in the \group\ that contains $P_t$ for every $t$.  As the algorithm is not initially rejected, $P_0$ is uniformly distributed in $\acceptable$, and thus is uniformly distributed in the \group\ containing $P_0$, providing the base case. Assume that $t\ge 1$,  and $P_{t-1}$ is uniformly distributed in $\state_{i}$, where $i=L(P_{t-1})$. Recall that $S^-_{L,A}(P)$ is the set of pairings that can be switched to $P$ via a type L switching (and recall that all type L switchings are of class A). By~\eqn{transitionprob}, for every $P\in \state_{i-1}$, 
\[
\pr(P_t=P)=\sum_{P'\in S^-_{L,A}(P)} \pr(P_{t-1}=P') \frac{\LB_{A}(i-1)}{\UB_{L}(i) b_{A}(P)}, 
\]
as every $P'\in S^-_{L,A}(P)$ must be in $\state_i$. Since $P_{t-1}$ is uniformly distributed in $\state_i$ by the inductive hypothesis, we have
$\pr(P_{t-1}=P')=1/|\state_i|$. On the other hand, $b_{A}(P)=|S^-_{L,A}(P)|$ as all class A switchings are of type L. Immediately we have that  
\[
\pr(P_t=P) =\frac{\LB_{A}(i-1)}{ |\state_i|\UB_{L}(i)},
\]
which is independent of $P$. So $P_t$ is uniformly distributed in $\state_{i-1}$. Inductively, $P'$ is uniformly distributed in $\state_0$, which is $\Phi'$.\qed
\end{proof}

\begin{lemma}\lab{lem:lrej}
The overall probability that either an f- or b-rejection occurs in phase 1 is $O((d^2+d\log^2 n)/n)$.
\end{lemma}
 \begin{proof}
We estimate the probability that an f-rejection occurs at a given step.
 Assume that $P\in\state_i$. Then,
 $f_L(P)=\UB_L(i)-X(P)$, where $X(P)$ is the number of choices of a loop and the two pairs $\{3,4\}$ and $\{5,6\}$ so that 
 \begin{enumerate}
 \item[(a)] $\{3,4\}$ or $\{5,6\}$ is a loop or is contained in a multi-edge --- at most $2\cdot 2iM_1(2i+4B_D+6B_T)$;
 \item[(b)] $v_1\in\{v_2,v_3,v_4,v_5\}$ --- at most $4\cdot 2id M_1$;
 \item[(c)] $v_1$ is adjacent to $v_2$ or $v_4$; or $v_3$ is adjacent to $v_5$ --- at most $2\cdot 2i M_1 d^2+2i M_1d^2$.
 \end{enumerate} 
  
 % The number of choices for (a) is $2\cdot 2iM_1(2i+4B_D+6B_T)$.
 
  %The number of choices for (b) is at most $4\cdot 2id M_1$.
  
  % The number of choices  for (c) is at most $2\cdot 2i M_1 d^2+2i M_1d^2$.
      Hence, $X(P)\le 2iM_1(4i+8B_D+12B_T+4d+3d^2)$.
   Then, the probability that $P$ is f-rejected is
   \[
   1-\frac{f_L(P)}{\UB_L(i)}=\frac{X(P)}{\UB_L(i)}=\frac{4i+8B_D+12B_T+4d+3d^2}{M_1}\le \frac{4i_1+8B_D+12B_T+4d+3d^2}{M_1}.
   \]
  As the number of loops decreases by one in each step and $L(P)\le i_1=B_L$, the total number of steps in phase 1 is at most $i_1$. Hence, the overall probability of an f-rejection in phase 1 is at most
  \[
  \frac{i_1(4i_1+8B_D+12B_T+4d+3d^2)}{M_1}.
  \]
  By~\eqn{initialbounds}, this is at most $O((d^2+d\log^2 n)/n)$. %\jcomk{We can work out the coefficient here after we agreed with the definition in~\eqn{initialbounds}. This can wait until the end.}
  
On the other hand, to assess b-rejection, assume $P\in\state_i$. It is obvious that $b_A(P)\le M_2M_1$. Hence, by~\eqn{loopbound} the probability that $P$ is b-rejected is
   \[
   1-\frac{\LB_A(i)}{b_A(P)}\le \frac {2dM_1(3i+6B_D+9B_T+3d+d^2)}{M_2M_1- 2dM_1(3i+6B_D+9B_T+3d+d^2)}. 
   \]
   As $i\le i_1=B_L$ and by~\eqn{initialbounds},
   $2dM_1(3i_1+6B_D+9B_T+3d+d^2)<M_2M_1/2$ for all sufficiently large $n$. Therefore, the probability that $P$ is b-rejected is at most
   \[
   \frac {4(3i_1+6B_D+9B_T+3d+d^2)}{(d-1)n}.
   \]
   As there are at most $i_1$ steps, the overall probability that a b-rejection occurs in phase 1 is at most 
   \[
   \frac {4i_1(3i_1+6B_D+9B_T+3d+d^2)}{(d-1)n}=O((d^2+d\log^2 n)/n).\qed
   \]
 
\end{proof}

\subsection{Phase 2: triple edge reduction}
\lab{triplereduction}

In this phase, we start with a pairing $P_0$ uniformly distributed in $\Phi_1=\{P\in\acceptable:\ L(P)=0\}$, which is a valid assumption  at the start of phase 2 by Lemma~\ref{lem:lunif}. We will use one type of switching, called type T, defined as follows.
\ss

\no
{\em Type T switching.} Let $P$ be a pairing containing at least one triple edge. Pick a triple edge from $P$ and label its end points as in Figure~\ref{f:triple}. Pick another three pairs $x$, $y$ and $z$, all of which represent single edges, and label their end points as in Figure~\ref{f:triple}. A type T switching manipulates the pairs as in the figure. This switching is valid only if the eight vertices involved as in the figure are all distinct, and there are no new multi-edges or loops created by its action.  All type T switchings are of class A. 
 \begin{figure}[htb]

 \hbox{\centerline{\includegraphics[width=9cm]{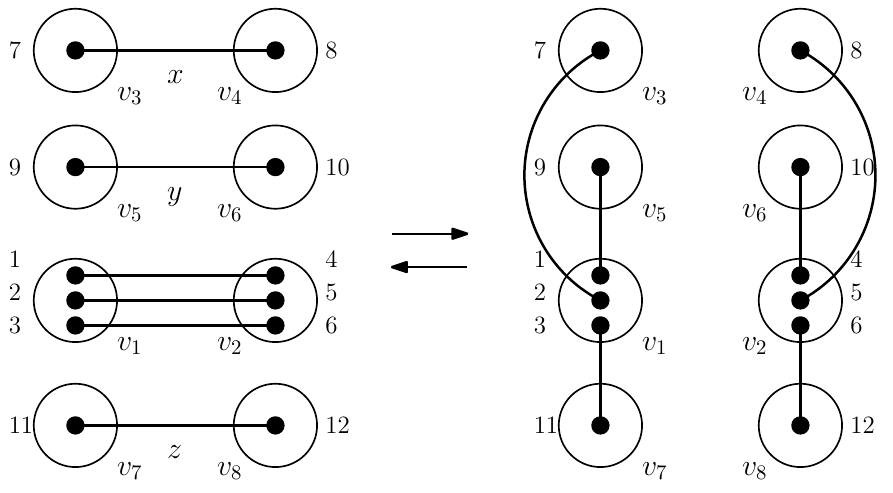}}}
 \caption{\it  type T switching}

\lab{f:triple}

\end{figure}

 In this phase, $\imax$ is set equal to $B_T$ defined in~\eqn{initialbounds}, and hence $\Phi_1=\cup_{0\le i\le \imax}\state_i$ where $\state_i$ is the set of pairings in $\Phi_1$ containing exactly $i$ triple edges. As only one type of switching is involved, we set $\rho_{T}(i)=1$ for every $1\le i\le \imax$.

Define 
\bea
\UB_{T}(i)&=&12 i M_1^3\non\\
\LB_{A}(i)&=&M_3^2-4M_3d^2(6B_D+9i+4d+d^2).\lab{triplebound}
\eea

\begin{lemma}
During phase 2, for each $P\in \state_i$, $f_{T}(P)\le \UB_{T}(i)$ and $b_A(P)\ge \LB_{A}(i)$.
\end{lemma}
\begin{proof}
The upper bound is trivial. There are $i$ ways to choose a triple edge and then $12$ ways to label their end points and then at most $M_1^3$ ways to choose the other $3$ pairs. 

For the lower bound, recall that for any $P\in\state_i$, $b_A(P)$ is the number of type T switchings that produce $P$. To count these, we choose two 3-stars and label their end points as on the right side of Figure~\ref{f:triple}. The number of such choices is $M_3^2$. Some of these choices are excluded:
 \begin{enumerate}
\item[(a)] one of the 3-stars contains a double edge or a triple edge --- at most $6M_3(4B_Dd^2+6id^2)$;
\item[(b)] the two 3-stars share a vertex --- at most $16M_3d^3$;
\item[(c)] there is an edge between $v_{2i-1}$ and $v_{2i}$ for some $i=1,2,3,4$ --- at most $4M_3d^4$.
\end{enumerate} 
 % The number of choices for (a) is at most $6M_3(4B_Dd^2+6id^2)$.
  
 % The number of choices for (b) is at most $16M_3d^3$.
  
  %The number of choices for (c) is at most $4M_3d^4$.
  
  Hence, $b_A(P)\ge M_3^2-4M_3d^2(6B_D+9i+4d+d^2)=\LB_A(i)$. \qed
   \end{proof}
\ss

Analogous to the proof of Lemma~\ref{lem:lunif} we have the following lemma, where $\Phi''=\{P\in \Phi_1:\ T(P)=0\}$.
\begin{lemma}\lab{lem:tunif}
Assume that no rejection occurs initially or during phase 1 or 2. Then, the output of phase 2 is a pairing uniformly distributed in $\Phi''$.
\end{lemma}

\begin{lemma}\lab{lem:trej}
The probability that either an f-rejection or a b-rejection occurs during phase 2 is $O(d^4/n^2+d\log^2 n/n)$.
\end{lemma}
 \begin{proof}
 We estimate the probability of an f-rejection in each step. Assume $P\in \state_i$ and recall the type T switchings in Figure~\ref{f:triple}; $f_{T}(P)=\UB_T(i)-X(P)$ where $X(P)$ is the number of choices of the triple edge and pairs such that
 \begin{enumerate}
 \item[(a)] $x$ or $y$ or $z$ is contained in a multi-edge --- at most  $3\cdot 12i M_1^2(4B_D+6i)$;
 \item[(b)] the triple edge is adjacent to $x$ or $y$ or $z$ --- at most $12\cdot 12id M_1^2$;
 \item[(c)] $v_1$ is adjacent to $v_3$ or $v_5$ or $v_7$; or $v_2$ is adjacent to $v_4$ or $v_6$ or $v_8$ --- at most $6\cdot 12id^2 M_1^2$.
 \end{enumerate} 
 %The number of choices for (a) is at most $3\cdot 12i M_1^2(4B_D+6i)$.
 
% The number of choices for (b) is at most $12\cdot 12id M_1^2$.
 
 %The number of choices for (c) is at most $6\cdot 12id^2 M_1^2$.
 
 Hence, $X(P)\le 72iM_1^2(2B_D+3i+2d+d^2)$. So, using $i\le i_1=B_T$,
 the probability that $P$ is f-rejected is 
 \[
 1-\frac{f_T(P)}{\UB_T(i)}=\frac{X(P)}{\UB_T(i)}\le \frac{72iM_1^2(2B_D+3i+2d+d^2)}{12 i M_1^3}\le \frac{6(2B_D+3i_1+2d+d^2)}{M_1}.
 \]
 As $T(P_0)\le i_1$, the number of steps in phase 2 is at most $i_1$ and hence, by~\eqn{initialbounds}, the overall probability of an f-rejection in phase 2 is at most
 \[
 \frac{6i_1(2B_D+3i_1+2d+d^2)}{M_1}=O(d^4/n^2+d\log^2 n/n).
 \]

On the other hand,
 %Simply remove this: Next we estimate the probability of a b-rejection in a step. 
for $P\in\state_i$, $\LB_A(i)\le b_A(P)\le M_3^2$. Using~\eqn{triplebound}, the probability that $P$ is b-rejected is
 \[
 1-\frac{\LB_A(i)}{b_A(P)}\le \frac{4d^2(6B_D+9i+4d+d^2)}{M_3-4d^2(6B_D+9i+4d+d^2)}.
 \]
 It is easy to see that $4M_3d^2(6B_D+9i_1+4d+d^2)\le M_3^2/2$ for all sufficiently large $n$. Hence, the above probability is at most 
 \[
 \frac{8d^2(6B_D+9i_1+4d+d^2)}{M_3}=O(d/n+\log n/dn).
 \]
 So the overall probability of a b-rejection in phase 2 is at most 
 $i_1\cdot O(d/n+\log n/dn)=O(d^4/n^2+d\log^2 n/n)$.\qed
   \end{proof}

\section{ \NameB\ and  proof of Theorems~\ref{thm:uniform},~\ref{t:complexity} and~\ref{thm:approx}}
\lab{sec:main}

 \NameA\ possesses several complicated features to achieve uniformity of the output distribution in each phase. Some of these involve time-consuming computations such as finding the probability of a b-rejection. By omitting some of these features, we obtain a simpler sampler \NameB.

\NameB\ starts by repeatedly generating a random pairing $P\in \Phi$ until $P\in \acceptable$ with $\gamma=1$ (or any positive number). Then, \NameB\ sequentially enters the three phases for reduction of loops, triple edges and double edges. In each phase, there will be only one type of switchings: types L and T for phases 1 and 2 respectively and type I for phase 3. Each phase consists of a sequence of switching steps in which a random switching is chosen and performed. Output the resulting pairing if it is in $\state_0$. 

In other words, \NameB\ is a simplification of \NameA\ by omitting f- or b-rejections in each phase, and by discarding type II switchings in phase 3. Therefore, there is no need to choose switching types in each switching step as in \NameA. As a result, there is no need of pre-computation for solving~\eqn{recx}--\eqn{atmost1}, and there is no t-rejection any more.

First we prove Theorem~\ref{thm:approx}, assuming Theorem~\ref{thm:uniform}.

%\ncz{ I have omitted the following details as it's easy. Also I don't know what you mean by 'some probability measure $\pi$' - measure on what events, and is it really an arbitrary measure?  Also it sneaks in the fact that f and b rejections are rare, and this is spelt out better below. I've replaced the following.  Assume the distribution of the output of the uniform sampler is $\pi_0$ and that of the approximate sampler is $\pi_1$. Then, $\pi_1=(1-\eps_n) \pi_0+\eps_n \pi$ for some probability measure $\pi$, where $\eps_n$ is the probability that any type of f or b-rejection occurs. It follows quickly that for any event $A$, $\pr_{\pi_1}(A)=(1-\eps_n)\pr_{\pi_0}(A)+O(\eps_n)=\pr_{\pi_0}(A)+o(1)$ as $\eps_n=o(1)$. By Theorem~\ref{thm:uniform}, $\pi_0$ is the uniform distribution over all $d$-regular graphs. This confirms that the total variation distance between $\pi_1$ and the uniform distribution is $o(1)$.}
\smallskip

\no {\bf Proof of Theorem~\ref{thm:approx}.\ }
 The probability of \NameA\ terminating with an f-rejection, a t-rejection of a b-rejection is $o(1)$ by Lemmas~\ref{lem:drejt},~\ref{lem:drejf} and~\ref{lem:drejb} respectively in phase 3, and similarly for the other phases.  The probability of \NameA\ performing a type II switching in phase 3 is $o(1)$ by Lemma~\ref{lem:steps} and the fact that $\rho_{II}(i)=O(d^2/n^2)$ as specified just after Lemma~\ref{lem:solution}. So, conditioning on the two algorithms generating the same initial pairing with no initial rejection, in an event with probability $1-o(1)$, \NameB\ has the same output as \NameA, which is uniformly distributed. In the remaining cases, an event of conditional probability $o(1)$, either some rejection occurs in \NameA\ and \NameB\ carries on regardless, or a type II switching is performed in \NameA\  and \NameB\ produces some output (exactly what is irrelevant). Bearing in mind that initial rejection occurs with probability bounded away from 1 by Lemma~\ref{l:initial}, it follows that    the total variation distance between the output of \NameB\  and the uniform distribution is $o(1)$.
  
 It is easy to see that uniformly generating a pairing $P\in\Phi$ takes  $O(dn)$ steps, and so does verification of $P\in\acceptable$ with $\gamma=1$, when using appropriate data structures. By Lemma~\ref{l:initial}, $P\in\acceptable$ with probability $1/2+o(1)$, and so it takes $O(dn)$ steps in expectation until a random $P\in\acceptable$ is generated. 
The total number of switching steps in all phases is bounded by $O(B_L+B_D+B_T)=O(d^2+\log n)$ in expectation and each switching step takes $O(1)$ unit of time. Thus, the total time complexity of \NameB\ is $O(dn+d^2+\log n)=O(dn)$ in expectation. \qed\ss

 Now we turn to \NameA. We first prove the uniformity  of the  output distribution.
 
 \smallskip

\no {\bf Proof of Theorem~\ref{thm:uniform}.\ } The uniform distribution of the output of \NameA\ is guaranteed by Lemmas~\ref{lem:uniform} and~\ref{lem:tunif}, since the parameters $\rho_{\tau}(i)$ are set for \NameA, just after Lemma~\ref{lem:solution},  using a solution $(\rho^*_{\tau}(i), x^*_i)$ of the system~\eqn{recx}--\eqn{atmost1}. (Recall that if the system had no solution, phase 3 was made redundant by requiring $D(P)=0$ in the definition of $\acceptable$ in Section~\ref{sec:A}.) \qed \ss
 
    Finally, we estimate the time complexity of \NameA.
    Implementing a switching step as described in Section~\ref{sec:framework} would apparently require knowledge of $f_{\tau}(P)$ and $b_{\alpha}(P)$. However, we first observe that there is no need to compute 
$f_{\tau}(P)$. This number is only used to compute the f-rejection probability at substep (iii) of a switching step. Take the type I switching in phase 2 as an example. Rather than computing $f_{I}(P)$, we can just randomly pick a double edge and label its ends, and then uniformly at random pick two pairs, with repetition, from the total $M_1/2-2i$ pairs other than those contained in a double edge, and then randomly label their ends. The total number of choices is exactly $\UB_{I}(i)$. Perform the switching if it is valid, and perform an f-rejection otherwise. The probability of an f-rejection is then exactly $1-f_{I}(P)/\UB_{I}(i)$. It is easy to see that this holds for all other types of switchings, by the way we defined each $\UB_{\tau}(i)$.   

\smallskip

\no {\bf Proof of Theorem~\ref{t:complexity}.\ } We have shown that   generating a pairing  $P\in\Phi$ uniformly at random takes $O(dn)$ steps. By Lemma~\ref{l:initial}, $P\in\acceptable$ with probability $\gamma/(1+\gamma)+o(1)$, and so it takes $O(dn/\gamma)$ steps in expectation until a random $P\in\acceptable$ is generated. To set the parameters   $\rho_{\tau}(i)$, we need to compute a solution to system~\eqn{recx}--\eqn{atmost1}. Using the computation scheme in Lemma~\ref{lem:solution}, the computation time is $O(B_D)=O(d^2)$, as remarked below Lemma~\ref{lem:solution}.

To generate a random graph, we repeat \NameA\ until there is a run in which rejection does not occur. By Lemmas~\ref{lem:lrej},~\ref{lem:trej},~\ref{lem:drejt},~\ref{lem:drejf} and~\ref{lem:drejb}, the probability that a rejection  occurs during any of the three phases is $o(1)$.  Thus, it is sufficient to bound the time complexity of each phase in one run of \NameA. 
  Each phase consists of a sequence of switching steps, and as observed above, we only need to consider the computation time of $b_{\alpha}(P)$.
 Phase 3 is the most crucial since it runs for the greatest number, $O(d^2)$,  of switching steps in expectation. The same ideas suffice easily for the other phases. An implementation that does the job is described in~\cite[Theorem 4]{MWgen}. Briefly, it goes as follows. Data structures are maintained that require an initial computation time of $O(nd^3)$ for the input pairing.  These data structures record, for instance, how many 3-paths join each given pair of vertices. (Note that the number of nonzero instances of this is $O(nd^3)$.)  They also record how many 3-paths join a pair of vertices using   a given point, and some other counts  relating to 2-paths, 3-paths starting with a double edge,  {\em pairs} of 3-paths joining each two given vertices, and similar. These data structures can be initialised in time $O(nd^3)$ by starting at each vertex in turn and investigating each 3-path leading from it. Moreover, they can be updated in time $O(d^2)$ for each switching step. This is because each switching step alters $O(1)$ pairs, each of which would affect $O(d^2)$ of the numbers being recorded (e.g.\ the number of 3-paths between two vertices). As the total number of switching steps in this phase is $O(d^2)$ in expectation, this requires $O(d^4)$ steps in expectation throughout all the switchings of the phase.   These  data structures are used for implementing a scheme explained in~\cite{MWgen}  to calculate the number of switchings of  class A that could produce the pairing $P'$ that was actually produced in  a switching step.  For a pair of 2-paths not to appear as those in the right hand side of Figure~\ref{f:doubleI}, there must be some ``coincidence'' occurring: either one of the edges shown is a double edge, or there is a coincidence of vertices, or some edge joins $u_2$ to $v_2$, or similar. The number of pairs of 2-paths satisfying any prescribed set of coincidences can be calculated using the data structures, and then one proceeds by inclusion-exclusion. For instance (one of the ``hardest'' cases) is   when both $u_2v_2$ and $u_3v_3$ are edges. To count these configurations, we count pairs of 3-paths joining $u_1$ to $v_1$. Instances where these two 3-paths intersect accidentally are accommodated as part of the whole inclusion-exclusion scheme. A bounded number of cases need to be considered,  and each can be calculated in constant time from the appropriate data structures. (A very large bound on the number of cases is supplied in~\cite{MWgen},  but actually they can be described by approximately 20 different graphs showing the way that two 2-paths can intersect or have adjacent corresponding pairs of vertices,  not counting the options for where double edges may occur among the edges of the 2-paths.) As a result, the most time-consuming part is the initialisation of the data structure, which takes $O(nd^3)$ time. Maintaining the data structure in the subsequent switching steps takes only $O(d^4)=o(nd^3)$ time.  \qed

 \ss
 
 \no{\bf Computational note.\ } 
If the constraints in Lemma~\ref{lem:solution}  are not satisfied, the algorithm terminates whenever the number of double edges is greater than 0, according to the definition of $\acceptable$ in Section~\ref{sec:A}. For practical purposes it may be of interest to know when they cannot be satisfied for any choice of $\gamma>0$ and $\epsilon$, as the algorithm then becomes equivalent to {\em DEG} and it may take many repetitions before a successful output occurs. We note that, for $d>3$, the equation
$$
4 (d-1)^2(d-2)(d-3)=dn (dn-7d^2+7d-10   )
$$
is quadratic in $n$ with a  unique positive solution:
$$
x(d):=\frac72 d -\frac72 +5/d+(1/2d)\sqrt{65 d^4-210 d^3+461 d^2-412 d+196)}.
$$
   Letting $x(d)$ denote that solution, 
 there will clearly exist a choice of $\eps$ and $\gamma$ satisfying~\eqn{nLB1} and~\eqn{nLB2}, provided that $n>x(d)$. This bound and~\eqn{nLB3} are both satisfied if $n$ is at least $69,64,64,66,69$ for $d=4,5,6,7,8$ respectively. For such $n$, the extra power of the algorithm comes into play. Of course, the lower bound is asymptotically linear in $d$.

\end{document}